\RemoveFromHook{class/amsart/after}[firstaid/aliascounter]
\documentclass[10pt,dvipsnames]{amsart}

\usepackage{preamble}
\usepackage{tikz}

\title[Constructive Tchakaloff results and well-conditioned quadrature]{Constructive Tchakaloff results and well-conditioned quadrature through randomized least squares}
\author{Filip B\v{e}l\'ik}
\author{Akil Narayan}
\author{John Turnage}

\begin{document}
\begin{abstract}
We consider using randomized least squares to construct quadrature rules exact on a subspace of functions. Using new conditions that we call relative admissibility and reference weight concentration, we establish both that the quadrature weights from such a procedure concentrate close to their asymptotic values with prescribed and arbitrarily large probability, and this in turn provides useful finite-sample probabilistic bounds on the stability of the resulting quadrature rules for very general classes of possibly complex-valued functions. Our analysis both significantly generalizes the existing analysis of randomized least squares quadrature construction, and provides new bounds on stability for these rules. These results specialize to existence results for positive quadrature rules, when such rules can be theoretically expected. Because our procedures are formally algorithmic, our analysis is a substantive advance toward computationally constructive generalized Tchakaloff theorems.
\end{abstract}

\maketitle

{\footnotesize

}

\section{Introduction}

Quadrature/cubature rules are standard ingredients in computational mathematics for approximating integrals or more general functionals \cite{Cools1997,Brass2011}. We consider the generation of such quadrature rules, exact on a prescribed subspace, through randomized least squares. 

Given an \(N\)-dimensional real- or complex-valued vector space \(V\subset L^2_\mu(D)\) and a functional
\(\mathcal L\in V^*\), we seek nodes \(x_m\in D\) and weights \(w_m\) such
that
\begin{equation}\label{eq:Q-L-exact}
  \mathcal{Q}_M(v) \coloneqq \sum_{m=1}^M w_m v(x_m)=\mathcal L(v),
    \qquad \forall\;v\in V.
\end{equation}
The desire for exactness on $V$ can be motivated through Lebesgue-type
inequalities that quantify accuracy of such $V$-exact rules through best
approximation numbers over $V$, cf. \eqref{eq:quad-lebesgue}. Such inequalities
are particularly powerful when the quadrature weights are real and positive,
as such weights avoid cancellation in the quadrature functional, give
immediate stability bounds in the uniform norm, and, for polynomial rules, are
closely connected with convergence of the resulting integral estimates on
smooth function classes.

In one dimension, positive quadrature with classical choices of $V$ is well understood.
Gaussian quadrature rules~\cite{Golub1969}, for example, are built from
orthogonal polynomial roots and have positive weights for broad classes of
positive measures; Clenshaw--Curtis quadrature~\cite{Clenshaw1960} is another
classical construction.  In several variables, especially on general domains
or for non-product measures, the situation is more difficult.
Tensor-product, sparse-grid constructions, and quasi-Monte Carlo rules 
\cite{Novak1999,Smolyak1963,Gerstner1998,dick_digital_2010,lemieux_monte_2009,niederreiter_random_1992}
are effective under suitable product structure, but each has drawbacks. For
example, product rules have large cardinality in high dimension, sparse grids
can have negative weights, and quasi-Monte Carlo rules are generally not exact
on subspaces and need not have high order accuracy for smooth functions. While
each of these approaches has practical computational underpinnings, each of
them also largely require product structure for $D$.

Methods to compute quadrature rules on unstructured domains are less common.
This search for positive quadrature rules can be motivated by Tchakaloff's
theorem, which addresses the existence of finitely supported positive measures
(i.e., positive quadrature rules) matching the moments of a given measure on a
finite-dimensional function space \cite{Tchakaloff1957,Davis1967,Bayer2006}.
Recent results have established the existence of positive quadrature in
substantial generality \cite{schafer_2025}, but do not provide a constructive
procedure to compute these rules.  One class of usually deterministic
computational methods is the family of optimization-based searches for cubature
rules \cite{Jakeman2017,Keshavarzzadeh2018,Keshavarzzadeh2018a} and
sample-based constructions~\cite{Bos2018b}.  These methods are useful in
practice, but positivity and exactness are not automatic consequences of the
approximation procedures used to construct them. The alternative that we
explore is the approach of constructing quadrature through randomized least
squares, as pioneered in \cite{migliorati2022stable}, although related work in
more specialized settings has appeared earlier \cite{Huybrechs2009,Wilson1970}.

The central questions we focus on are how to construct such $V$-exact rules
from random samples, how their stability depends on the sampling law, and when
it is possible to assert that positive-weight rules exist. In particular, we
establish new, relatively permissive assumptions that ensure that a trial-based
randomized least squares procedure can achieve quadrature rules whose weights
lie close to their asymptotic value, along with guarantees on the conditioning
of such quadrature rules. These results, specialized to the case when positive
quadrature is a reasonable goal, provide an algorithmic template for Tchakaloff
theorems. Like \cite{migliorati2022stable}, our randomized results hold with
probability \(1-\xi\), for arbitrary \(0<\xi<1\). In contrast to
\cite{migliorati2022stable}, our guarantees are more general for existence and
stability of quadrature rules. We investigate stable (and when appropriate,
positive) quadrature in situations when the results of
\cite{migliorati2022stable} are not applicable; in exchange for this
generality, our required sample size may be specified only implicitly through
a new  ``relative admissibility'' condition. In contrast to the abstract
existence results in \cite{schafer_2025}, our results are slightly less
general, applying in situations when we can excise certain pathologies; in
exchange we gain constructivity through an explicit least squares methodology.

\subsection{Summary of contributions}
Our scope is limited to the construction of quadrature rules through randomized
least squares, cf. \Cref{ssec:procedure}.  Within this scope, we make several
theoretical contributions under appropriate assumptions.
\begin{itemize}
  \item \textit{Quadrature weight concentration}: We show that a finite-$M$ quadrature rule exists whose weights concentrate around their ``reference'' values. See \Cref{thm:finite-sample-weight-approximation}.
  \item \textit{Stable quadrature rules}: The resulting quadrature rules have bounded condition numbers with respect to functions from appropriate weighted spaces. See \Cref{thm:stable-intermediate-rule}.
  \item \textit{Generalized Tchakaloff theorems}: When $(V,\mathcal{L})$ satisfy certain positivity-type properties, one expects that positive quadrature rules can be identified. (To contrast, if $\mathcal{L}$ were negative integration, one does not expect that positive weights would be possible or desirable.) Under these positivity properties, we combine the previous results with a deterministic Carath\'eodory--Steinitz pruning
argument to show the existence of positive quadrature rules satisfying \eqref{eq:Q-L-exact} with interpolatory size ($M \leq N$). See \Cref{thm:l2-tchakaloff,thm:l2-tchakaloff-conjugate}.
\end{itemize}
We emphasize that our results inherit the classical hallmark of randomness: The
dimension, $d = \dim D$, is not formally explicit, but rather implicit in some of
our quantitative constants. The previously mentioned assumptions we require to
conclude our results involve the Riesz representer $L \in V$ of $\mathcal{L}$,
and the $L^2$-Christoffel-type function of
$(V,\mu)$:
\begin{align*}
  \left\langle L, v\right\rangle_{L^2_\mu} \coloneqq \mathcal{L}(v) \;\;\forall\; v \in V, \hskip 20pt
  k(x) \coloneqq \sup_{v \in V\backslash\{0\}} \frac{|v(x)|^2}{\|v\|_{L^2_\mu}^2}.
\end{align*}
Our core assumption, ``relative admissibility'', informally requires that
$k(x)/L(x)^2$ can be large only with sufficiently small probability, when $x$
is distributed under the law of how samples are drawn. See
\Cref{ssec:assumption} for the precise conditions. (In contrast, the main
characterization of positive quadrature in \cite{migliorati2022stable} requires
the much stronger conditions $L \equiv 1$ and $k$ being uniformly bounded with
probability 1.) The stability results in \Cref{thm:stable-intermediate-rule}
require yet another (different) assumption of ``reference weight
concentration'', but we show in \Cref{sssec:mixture-sampling} how a sampler that
achieves only the first relative admissibility assumption can be augmented in a
computationally feasible manner to achieve both assumptions.

All our proofs are algorithmic in nature. The most literal implementation is a
trial-based procedure: sample an \(M\)-point least-squares rule, test whether
the weights are stable and/or positive; prune when successful to achieve an
$N$-point Tchakaloff rule. For the sample sizes identified by our results,
success is achieved with high probability. Repeated trials for these sample
sizes remain a practical option for checking success.

We re-emphasize that our results are distinct and new, but are similar to
existing ones in methodology \cite{migliorati2022stable} and in impact
\cite{schafer_2025}. At a high level, our setup is considerably more general
than in \cite{migliorati2022stable} (e.g., we allow infinite measures $\mu$,
unbounded Christoffel-type functions $k$, relatively flexible sampling choices,
general integration-type functionals and approximation weights, and
complex-valued functions), our relative admissibility assumption is
substantially weaker than similar assumptions in \cite[Theorem
6]{migliorati2022stable}, and we provide different types of probabilistic
guarantees. However, our setup and assumptions are somewhat stronger than
required for the abstract existence results in \cite{schafer_2025}, but our
existence results are explicitly constructive.  We provide more technical
comparisons and discussion that distinguish our results from these existing
ones in \Cref{ssec:assumptions-discussion}.

The remainder of the paper is organized as follows.
\Cref{sec:preliminaries} introduces the notation, admissibility
hypotheses, and main results.  \Cref{sec:ls,sec:posls,subsec:reduced} all focus on providing proofs for the main results stated in \Cref{sec:preliminaries}. In particular, \Cref{sec:ls} constructs the
least-squares quadrature rule and derives the limiting form of its weights.
\Cref{sec:posls} proves the finite-sample concentration estimates for
the weights and derives the positive intermediate rule.
\Cref{subsec:reduced} gives the Carath\'eodory--Steinitz pruning
argument and its consequences for the constructive Tchakaloff theorem.

We provide a numerical example in \Cref{sec:numerics},
discloure statements in \Cref{sec:disclosures},
and \Cref{sec:conclusion} concludes.

\section{Setup and main results}

\label{sec:preliminaries}
This section introduces the notation and hypotheses used throughout the
paper and states the main results.  The construction proceeds by first forming
an exact \(M\)-point least squares quadrature rule, then using finite-sample
weight estimates to obtain stability and/or positivity on an
event of probability at least \(1-\xi\), for a prescribed \(0<\xi<1\).
Deterministic Carath\'eodory--Steinitz pruning may
be applied to reduce the rule to interpolatory size, which maintains positivity
if the original $M$-point rule is positive. The definitions below are organized
to make these steps precise. 

\subsection{Notation}
\label{ssec:notation}

We will write finite dimensional vectors in lowercase boldface as $\mathbf{x}\in\mathbb{R}^d$ or $\mathbf{y}\in\mathbb{C}^d$ and matrices in uppercase boldface as $\mathbf{A}\in\mathbb{R}^{N\times M}$ or $\mathbf{B}\in\mathbb{C}^{N\times M}$. We will primarily work with complex-valued spaces as most results generalize or specialize to real-valued spaces. For $z = a + ib \in \mathbb{C}$ with $a,b\in\mathbb{R}$, we write $\overline{z}=a-ib$ to denote its complex conjugate and note that $\overline{a}=a$ for all $a\in\mathbb{R}$. We make careful use of both the transpose and adjoint operators defined respectively as follows for $\mathbf{B}\in\mathbb{C}^{N\times M}$,
\[(\trans{\mathbf{B}})_{i,j} = \mathbf{B}_{j,i},\quad (\mathbf{B}^*)_{i,j} = \overline{\mathbf{B}_{j,i}},\quad 1\leq i\leq M,\;1\leq j\leq N.\]
Similar transpose and adjoint logic is applied to vectors. The purpose of the inclusion of both is that while most linear algebra respects the complex inner-product, quadrature rules are naturally written with real inner-products. 

Let \(\mu\) be a positive measure with closed support \(D\) in a Banach space. 
We write
\begin{equation}
\label{eq:U-def}
    \mathcal U(D)
    \coloneqq
    \left\{
        u:D\to\mathbb C:
        u \text{ is \(\mu\)-measurable}
    \right\},
\end{equation}
over which pointwise evaluation is well-defined.
We work in the complex Hilbert space $L^2_\mu(D)$ defined as
\begin{align}
\label{eq:l2mu-def}
    L^2_\mu(D)
    &=
    \left\{
        u \in \mathcal{U}(D) :
        \int_D |u(x)|^2 \dx\mu(x)<\infty
    \right\} \big/ \sim_{\mu},
\end{align}
where $u\sim_\mu v$ means $u=v$ $\mu$-almost everywhere, and the inner-product and norm are defined respectively as
\begin{equation}
\label{eq:l2mu-normip-def}
    \langle u,v\rangle_{L^2_\mu(D)}
    =
    \int_D\overline{u(x)}v(x)\,d\mu(x),
    \qquad
    \|u\|_{L^2_\mu(D)}^2
    =
    \langle u,u\rangle_{L^2_\mu(D)}.
\end{equation}
To reduce notational clutter, we will write
\[
    \|\cdot\|=\|\cdot\|_{L^2_\mu(D)},
    \qquad
    \langle\cdot,\cdot\rangle
    =
    \langle\cdot,\cdot\rangle_{L^2_\mu(D)}.
\]
For any positive integer \(P\), we write
\begin{equation}
\label{eq:index-set-def}
    [P]
    \coloneqq
    \{1,\ldots,P\}.
\end{equation}

\subsubsection{The subspace \(V\)}
\label{sssec:preliminaries-V}

Let \(V\subset L^2_\mu(D)\) be a finite-dimensional subspace with
\(
    N=\dim(V) < \infty.
\)
Choose an \(L^2_\mu\)-orthonormal basis of \(V\), and fix a measurable representative \(\varphi_n:D\to\mathbb C\) of each basis element. We henceforth identify the equivalence class \(\sum_{n=1}^N c_n[\varphi_n]_\mu\in V\) with the pointwise-defined function
\[
x\longmapsto\sum_{n=1}^N c_n\varphi_n(x).
\]With this convention,
\begin{equation}
\label{eq:V-basis-def}
V=\operatorname{span}\{\varphi_n\}_{n=1}^N\subset\mathcal U(D),
\end{equation}
and all point evaluations below refer to this fixed realization of \(V\).

We construct quadrature rules that are exact on \(V\) using $L^2_\mu$-least
squares problems involving point evaluation data. When $\mu$ is a finite
measure so that sampling from it is possible, a corresponding notion of least
squares stability with such data that surfaces in analysis is the
Christoffel-type quantities,
\begin{equation}
\label{eq:k-def}
    k(x)
    =
    \sup_{v\in V\setminus\{0\}}
    \frac{|v(x)|^2}{\|v\|^2},
    \qquad
    K= \mathop{\mathrm{ess\,sup}}_{x \in D} k(x),
\end{equation}
which implicitly also depend on $\mu$.
Equivalently,
\[
    k(x)=\sum_{n=1}^N |\varphi_n(x)|^2.
\]
Thus, \(\sqrt{k(x)}\) is the operator norm of evaluation at \(x\) on the space \(V\).

Large values of \(k\) identify points at which functions in \(V\) can be large
relative to their \(L^2_\mu\)-norm, and the value of \(K\) enters sample
complexity estimates in least squares constructions.  

\subsubsection{The sampling measure $\rho$}\label{ssec:notation-sampling}
Since large values of \(K\) result in poor sample complexity estimates, we consider a change of measure technique wherein we still form approximations in $V \subset L^2_\mu(D)$, but generate samples from a different measure $\rho$. Let
\(\tau^2:D\to[0,\infty)\) be measurable and belong to $L^1_\mu(D)$, normalized so that
\(\|\tau^2\|_{L^1_\mu(D)} = 1\). We assume the following coverage condition:
\begin{equation}
\label{eq:sampling-coverage}
    k(x)=0
    \text{  for \(\mu\)-almost every  } x\in\{\tau^2=0\}.
\end{equation}
Equivalently, \(\tau^2(x)>0\) for \(\mu\)-almost every point at which
\(k(x)>0\). Thus, sampling may omit only a set on which every function in
\(V\) vanishes. We then define the probability measure $\rho$ by
\begin{equation}\label{eq:tau-def}
    \dx \rho=\tau^2 \dx \mu.
\end{equation}

Note that $\rho$ is always a probability measure even when $\mu$ is not a
finite measure. When sampling with $\rho$ but constructing $L^2_\mu$ least
squares approximations, the analogue of the stability parameter $K$ is,
\begin{equation}\label{eq:Q-def}
    Q
    =
    \mathop{\mathrm{ess\,sup}}_{x\sim\rho}
    \frac{k(x)}{\tau^2(x)}.
\end{equation}
Here and below, ratios of nonnegative quantities use the conventions
\(0/0=0\) and \(a/0=+\infty\) for \(a>0\). The essential supremum in
\eqref{eq:Q-def} is the relevant quantity for the concentration estimates;
a pointwise supremum would depend on the chosen pointwise representatives of
the functions in \(V\) and would impose a stronger condition.
Many results below assume \(Q<\infty\). This is only a \(\rho\)-almost-sure
bound and does not constrain \(\{\tau^2=0\}\), which is instead governed by
\eqref{eq:sampling-coverage}; common zeros of \(k\) and \(\tau^2\) are
permitted.
Under the $\rho$ sampling measure, it is the tail behavior of \(k\), together
with the change of measure Jacobian, \(\tau^2\), that controls the finite-sample estimates in
\Cref{sec:posls}. Because $\rho$ is a probability measure, every
sampling density satisfying \eqref{eq:sampling-coverage} obeys the lower
bound
\begin{align}
    \label{eq:Q-lower-bound}
    Q
    =
    \mathop{\mathrm{ess\,sup}}_{x\sim\rho}\frac{k(x)}{\tau^2(x)}
    \ge
    \int_{\{\tau^2>0\}} \frac{k(x)}{\tau^2(x)}\dx{\rho}(x)
    =
    \int_{\{\tau^2>0\}} k(x)\dx{\mu}(x)
    =
    \int_D k(x)\dx{\mu}(x)
    =
    N.
\end{align}
Equality (the minimum of $Q = N$) is attained by the particularly important choice of the \emph{induced distribution}
\begin{equation}\label{eq:induced_measure}
\dx \rho_{\mathrm{ind}} \coloneqq \tau_{\mathrm{ind}}^2
\dx \mu, \qquad 
\tau_{\mathrm{ind}}^2=k/N,
\end{equation}
for which \(k/\tau_{\mathrm{ind}}^2=N\) on \(\{k>0\}\), while both
numerator and denominator vanish on \(\{k=0\}\). Thus, the coverage
condition holds and \(Q=N\). We consider general samplers, possibly
$\tau \neq \tau_{\mathrm{ind}}$, for example because in some computational
scenarios only an approximation $\tau \approx \tau_{\mathrm{ind}}$ may be
available as a sampler.

\subsubsection{The functional \(\mathcal L\)}
\label{sssec:preliminaries-L}

Let \(\mathcal L\in V^*\) be a bounded linear functional on \(V\).  A
central example, whenever this is bounded on $V \subset L^2_\mu$, is the integral functional
\[
    \mathcal L(v)
    =
    \int_D v(x)\dx\mu(x),
    \qquad v\in V.
\]

The least squares weights will be expressed in terms of the action of
\(\mathcal L\) on a basis of \(V\).  It is therefore useful to identify
\(\mathcal L\) with its \(L^2_\mu\)-Riesz representer on \(V\).  Let
\(\{\varphi_n\}_{n=1}^N\) be an(y) orthonormal basis for \(V\), and define the
moment vector
\begin{equation}
\label{eq:moment-vector}
    \bs\ell
    =
    \trans{(\ell_1,\ldots,\ell_N)}\in\mathbb C^N,
    \qquad
    \ell_n=\mathcal L(\varphi_n).
\end{equation}
The \(L^2_\mu\)-Riesz representer of \(\mathcal L\) on \(V\) is the function
\(L\in V\) defined by
\begin{equation}
\label{eq:L-def}
    L(x)
    =
    \bs\ell^*\bs\varphi(x)
    =
    \sum_{n=1}^N \overline{\ell_n}\,\varphi_n(x),
    \qquad
    \bs\varphi(x)
    =
    \trans{(\varphi_1(x),\ldots,\varphi_N(x))}.
\end{equation}
Then, for every \(u\in V\),
\begin{equation}
\label{eq:L-integral}
    \mathcal L(u)
    =
    \langle L,u\rangle
    =
    \int_D \overline{L(x)}u(x)\,\dx \mu(x), \hskip 10pt
    \|\mathcal L\|_{V^*}=\|L\|=\|\bs\ell\|_2.
\end{equation}
Although \eqref{eq:L-def} is written in terms of a basis, \(L\) is
basis-independent. Of particular note is the special case when $\mathcal{L}$ is
integration and constant functions are in $V$ ($1 \in V$), as in that case we
have $L(x) = 1$. See \Cref{app:integral-representer} for a brief proof.

\subsubsection{Quadrature rules}
\label{sssec:preliminaries-quadrature}

Given \((\mu,V,\mathcal L)\), an \(M\)-point quadrature rule on
\(\mathcal U(D)\) that is exact on \(V\) is a functional of the form
\begin{equation}\label{eq:Q-V-exact}
    \mathcal Q_M:\mathcal U(D)\to\mathbb C,
    \qquad
    \mathcal Q_M(u)
    =
    \sum_{m=1}^M w_m u(x_m), \hskip 15pt
    \mathcal Q_M(v)=\mathcal L(v),
        \;\; \forall\;v\in V,
\end{equation}
where \(x_m\in D\) are the nodes and \(w_m\in\mathbb C\) are the weights.  Let \(\mathsf W:D\to(0,\infty)\) be a positive, measurable approximation weight and define the weighted $L^\infty$-type supremum space
\begin{equation}
\label{eq:weighted-sup-space}
    \mathcal{B}^\infty_{\mathsf W}(D)
    \coloneqq
    \left\{
        u\in\mathcal{U}(D):
        \|u\|_{\infty,\mathsf W}
        \coloneqq
        \|\mathsf W u\|_\infty
        =
        \sup_{x\in D}\mathsf W(x)|u(x)|
        <\infty
    \right\}.
\end{equation}
Here we take \(\mathcal{B}^\infty_{\mathsf W}(D)\) to consist of \(\mu\)-measurable,
pointwise-defined functions, rather than equivalence classes modulo
almost-everywhere equality, since quadrature rules involve point evaluation.
Thus, \(\mathcal{B}^\infty_{\mathsf W}(D)\subset\mathcal U(D)\).
We measure the stability of \(\mathcal Q_M\) by its operator norm on
\(\mathcal{B}^\infty_{\mathsf W}(D)\):
\begin{equation}
\label{eq:condition-number}
    \kappa_{\mathsf W}(\mathcal Q_M)
    \coloneqq
    \sup_{\substack{u\in \mathcal{B}^\infty_{\mathsf W}(D)\\
    \|u\|_{\infty,\mathsf W}\leq1}}
    |\mathcal Q_M(u)|
    \leq
    \sum_{m=1}^M
    \frac{|w_m|}{\mathsf W(x_m)}.
\end{equation}
The right-hand side is a coefficient bound. It is exact when the listed nodes
are distinct; at coincident nodes, the corresponding coefficients combine
before their absolute value is taken.
Indeed, \(\|u\|_{\infty,\mathsf W}\leq1\) is equivalent to
\[
    |u(x)|\leq\mathsf W(x)^{-1},
    \qquad x\in D,
\]
so \(\kappa_{\mathsf W}(\mathcal Q_M)\) is the largest possible value of
\(\left|\sum_m w_m u(x_m)\right|\) over functions satisfying this pointwise
bound. It therefore quantifies the amplification of perturbations measured in
the weighted supremum norm.

The choice \(\mathsf W\equiv1\) recovers the usual absolute condition number
\begin{equation}
    \label{eq:standard-condition-number}
    \kappa(\mathcal Q_M)\leq\sum_{m=1}^M|w_m|,
\end{equation}
for which we suppress the subscript. Whenever real-valued weights are expected,
positive rules are especially desirable because they avoid cancellation. In
the unweighted case, if \(w_m\geq0\) and \(1\in V\), then exactness on constants
gives
\[
    \kappa(\mathcal Q_M)
    =
    \sum_{m=1}^M w_m
    =
    \mathcal L(1).
\]

The weighted formulation is useful on unbounded domains, where elements of
\(V\), such as nonconstant polynomials, may have infinite supremum norm while
\(\mathsf Wv\) remains bounded. It therefore permits meaningful uniform
approximation and stability estimates for functions with controlled growth at
infinity. For the resulting error estimate, suppose that
\[
    V\subset \mathcal{B}^\infty_{\mathsf W}(D)
    \qquad\text{and}\qquad
    \mathcal L_{\mathsf W}\in
    \bigl(\mathcal{B}^\infty_{\mathsf W}(D)\bigr)^*,
    \quad
    \mathcal L_{\mathsf W}|_V=\mathcal L.
\]
Then exactness on \(V\) gives the weighted Lebesgue inequality
\begin{align}\label{eq:quad-lebesgue}
    |\mathcal Q_M(u)-\mathcal L_{\mathsf W}(u)|
    &\leq
    \left(
        \kappa_{\mathsf W}(\mathcal Q_M)
        +
        \|\mathcal L_{\mathsf W}\|_{\infty,\mathsf W}^*
    \right)
    e_{V,\infty,\mathsf W}(u),\\
    \label{eq:quad-pointwise-err}
    e_{V,\infty,\mathsf W}(u)
    &=
    \inf_{v\in V}\|u-v\|_{\infty,\mathsf W},
\end{align}
for every \(u\in \mathcal{B}^\infty_{\mathsf W}(D)\), where
\(\|\mathcal L_{\mathsf W}\|_{\infty,\mathsf W}^*\) denotes the operator norm
of \(\mathcal L_{\mathsf W}\) on this space.
A proof is given in \Cref{app:lebesgue}. Thus stable
rules (small or controlled $\kappa_{\mathsf W}(\mathcal Q_M)$) convert weighted best approximation estimates, such as weighted Jackson
inequalities~\cite{Jackson1982,Passow1970}, into quadrature error estimates; see, e.g.,
\cite{lubinsky2007survey}. Further results on stable quadrature can be found in
\cite{Brass2011,Watson1980,Ibrahimoglu2016}.

We call a quadrature rule \emph{interpolatory-size} if it is exact on \(V\) and
has at most \(N=\dim(V)\) active nodes. (``Active'' nodes are those with
nonzero weights.) If exactly \(N\) active nodes are present and the associated
evaluation matrix is nonsingular, this agrees with the usual interpolatory
construction based on cardinal functions.

For convenience, \Cref{tab:notation} summarizes the principal notation used
throughout the paper.

\begin{table}[t]
\small
\centering
\renewcommand{\tabcolsep}{0.28cm}
\renewcommand{\arraystretch}{1.18}
\begin{tabular}{@{}ccp{0.74\textwidth}@{}}
\toprule
\multicolumn{3}{@{}l}{\textbf{Ambient Setting}} \\
\midrule
\(\mathcal U(D)\) & \eqref{eq:U-def} & Pointwise-defined ambient space for quadrature inputs \\
\(D\), \(\mu\), \(L^2_\mu(D)\) & \eqref{eq:l2mu-def} & Support, measure, and the ambient Hilbert space \\
\(\langle\cdot,\cdot\rangle\), \(\|\cdot\|\) & \eqref{eq:l2mu-normip-def} & Ambient Hilbert space norm and inner product \\
\([P]\) & \eqref{eq:index-set-def} & Index set \(\{1,\ldots,P\}\) \\
\(V\), \(N\),\(\{\varphi_n\}_{n=1}^N\) & \eqref{eq:V-basis-def} & Approximation space, its dimension, and an \(L^2_\mu\)-orthonormal basis\\
\(k(x)\), \(K\) & \eqref{eq:k-def} & Christoffel-type quantity and its supremum, measuring point evaluation \\
\midrule
\multicolumn{3}{@{}l}{\textbf{Sampling, Functionals, and Quadrature}} \\
\midrule
\(\tau^2\), \(\rho\) & \eqref{eq:tau-def} & Sampling density with respect to \(\mu\) and associated probability measure \\
\(X\sim\rho\), \(\E_\rho\) & \eqref{eq:tau-def} & Generic \(\rho\)-distributed sample and expectation under \(\rho\) \\
\(Q\) & \eqref{eq:Q-def} & Parameter governing empirical Gramian concentration and sample complexity \\
\(\rho_{\mathrm{ind}}\), \(\tau_{\mathrm{ind}}^2\) & \eqref{eq:induced_measure} & Induced distribution and its density; this choice attains \(Q=N\) \\
\(\bs\ell\) & \eqref{eq:moment-vector} & Moment vector of \(\mathcal L\)\\
\(\mathcal L\), \(\|\mathcal L\|_{V^*}\) & \eqref{eq:L-integral} & Bounded linear functional on \(V\) and its operator norm \\
\(L\), \(\bs\varphi(x)\) & \eqref{eq:L-def} & \(L^2_\mu\)-Riesz representer of \(\mathcal L\) and point-evaluation vector of basis functions \\
\(M\), \(x_m\), \(w_m\) & \eqref{eq:Q-V-exact} & Number of sampled nodes, quadrature nodes, and quadrature weights \\
\(\mathcal Q_M\) & \eqref{eq:Q-V-exact} & \(M\)-point quadrature rule exact on \(V\) \\
\(\mathsf W\), \(\mathcal{B}^\infty_{\mathsf W}(D)\) & \eqref{eq:weighted-sup-space} & Approximation weight and weighted supremum space \\
\(\kappa_{\mathsf W}(\mathcal Q_M)\) & \eqref{eq:condition-number} & Weighted condition number used to quantify quadrature stability \\
\(\|\mathcal L_{\mathsf W}\|_{\infty,\mathsf W}^*\) & \eqref{eq:quad-lebesgue} & Operator norm of the extension \(\mathcal L_{\mathsf W}\) \\
\(e_{V,\infty,\mathsf W}(u)\) & \eqref{eq:quad-pointwise-err} & Best weighted uniform approximation error over \(V\) \\
\midrule
\multicolumn{3}{@{}l}{\textbf{Admissibility and Concentration}} \\
\midrule
\(\omega_{\mathcal L,\rho}(x)\) & \eqref{eq:reference-weight} & Reference weight governing the large-\(M\) limit of the least-squares weights \\
\(\cos\gamma(x)\) & \eqref{eq:alignment-angle} & Alignment factor between \(\mathcal L\) and point evaluation \\
\(p_J\) & \eqref{eq:relative-amplification-tail} & Tail probability used to characterize relative admissibility \\
\(A_{\mathcal L,\rho}(x)\) & \eqref{eq:absolute-amplification} & Absolute amplification factor \\
\(q_J\) & \eqref{eq:absolute-admissibility} & Tail probability used to characterize absolute admissibility \\
\(\mathsf K_{\mathcal L,\rho,\mathsf W}(x)\) & \eqref{eq:weighted-reference-magnitude} & Weighted magnitude of the reference weight \\
\(\|\mathsf K_{\mathcal L,\rho,\mathsf W}(X)\|_{\psi_1}\) & \eqref{eq:reference-weight-concentration} & Sub-exponential Orlicz norm used in reference-weight concentration \\
\midrule
\multicolumn{3}{@{}l}{\textbf{Real and Positive Quadrature}} \\
\midrule
\(V_{\mathbb R}\) & \eqref{eq:VR-def} & Real subspace of \(V\) used in the real-preserving and positive-quadrature discussion \\
\midrule
\multicolumn{3}{@{}l}{\textbf{Least-Squares Construction}} \\
\midrule
\(u_V^{(M)}\) & \eqref{eq:wls-functional} & Weighted discrete least-squares approximation in \(V\) \\
\(\mathbf G^{(M)}\), \(\mathbf c\) & \eqref{eq:normal-equations} & Empirical Gramian and least-squares coefficient vector \\
\(\mathbf u\) & \eqref{eq:V-W-def} & Sampled data vector \(\trans{(u(x_1),\ldots,u(x_M))}\) \\
\(\mathbf V\), \(\mathbf W\) & \eqref{eq:V-W-def} & Evaluation matrix and diagonal weight matrix \\
\(\mathbf w\) & \eqref{eq:w-vector} & \(M\)-point quadrature-weight vector \\
\bottomrule
\end{tabular}
\renewcommand{\arraystretch}{1}
\renewcommand{\tabcolsep}{12pt}
\caption{Notation used throughout this paper.}
\label{tab:notation}
\end{table}

\subsection{The main procedure: ``intermediate'' least squares quadrature}\label{ssec:procedure}
We have enough notation to describe the main procedure we consider and analyze: For some fixed $M$, let $\{x_m\}_{m \in [M]}$ denote a collection of iid samples distributed according to the sampling measure $\rho$. Given \(u\in\mathcal U(D)\), we define \(\mathcal Q_M(u)\) by having \(\mathcal L\) act on the (weighted) least squares approximation on the samples \(x_m\):
\begin{align}\label{eq:ls-quad-procedure}
  \mathcal{Q}_M(u) &= \mathcal{L}\left(u^{(M)}_V\right), & u^{(M)}_V &= \argmin_{v \in V} \sum_{m \in [M]} \frac{1}{M \tau^2(x_m)} \left| u(x_m) - v(x_m)\right|^2.
\end{align}
We assume above that the least squares solution $u^{(M)}_V$ is unique; this is guaranteed (say with high probability) by a sufficiently large choice of $M$, which is accounted for in our detailed analysis. As our use of the notation $\mathcal{Q}_M$ suggests, \eqref{eq:ls-quad-procedure} is actually the implicit formula for a quadrature rule of the form \eqref{eq:Q-V-exact}. In particular, there are weights $\{w_m\}_{m \in [M]}$, independent of $u$, such that $\mathcal{Q}_M(u)$ as defined above is equal to $\sum_{m} w_m u(x_m)$. See \Cref{ssec:ls-weights}, in particular \eqref{eq:w-vector}, for the explicit identification of weights from this problem. This idea and procedure is also the central formulation in \cite{migliorati2022stable}.

Our main theoretical advances describe sufficient conditions (assumptions) that guarantee when the procedure \eqref{eq:ls-quad-procedure} produces a quadrature rule with desirable properties.  

\subsection{Assumptions for finite-sample weight approximation}
\label{ssec:weight-approximation}

For a sampled node \(x_m\), the limiting (large-$M$), $M$-scaled least squares weight is
\begin{equation}
\label{eq:reference-weight}
    \omega_{\mathcal L, \rho}(x_m)
    \coloneqq
    \frac{\overline{L(x_m)}}{\tau^2(x_m)},
\end{equation}
see Lemma \ref{lemma:wm-asymptotic}, which shows, for each fixed \(m\), that
\(M w_m^{(M)}\to\omega_{\mathcal L,\rho}(x_m)\) almost surely as
\(M\to\infty\). We seek finite-sample control of
\begin{equation}
\label{eq:weight-approximation-error}
    \left|
        M w_m-\omega_{\mathcal L, \rho}(x_m)
    \right|,
\end{equation}
where \( (x_m, w_m) \) are weights and nodes associated to the quadrature rule \eqref{eq:ls-quad-procedure}.
The coverage condition \eqref{eq:sampling-coverage} implies that
\(L(x)=0\) for \(\mu\)-almost every \(x\in\{\tau^2=0\}\), since
\(|L(x)|\leq\|L\|\sqrt{k(x)}\). We set
\(\omega_{\mathcal L,\rho}(x)=0\) on this set. This convention has no
effect on the sampled weights because \(\rho(\{\tau^2=0\})=0\).
Our primary result concerns relative error, since relative accuracy smaller
than one preserves positivity when the limiting weights are nonnegative.
We also record a corresponding absolute-error condition. 

The next sections present the assumptions we require to control finite-sample fluctuations. These conditions are technical, and we provide a more complete discussion devoted to making them more transparent in \Cref{ssec:assumptions-discussion}.

\subsubsection{Relative admissibility}
\label{ssec:assumption}

Assume that \(\mathcal L\neq0\). For \(x\in D\) with \(k(x)>0\), define
\begin{equation}
\label{eq:alignment-angle}
    \cos\gamma(x)
    \coloneqq
    \left|
        \left\langle
            \frac{\bs\ell}{\|\bs\ell\|_2},
            \frac{\bs\varphi(x)}{\|\bs\varphi(x)\|_2}
        \right\rangle
    \right|
    =
    \frac{|L(x)|}
    {\|\mathcal L\|_{V^*}\sqrt{k(x)}}.
\end{equation}
If \(k(x)=0\), then \(L(x)=0\), and we set
\(\cos\gamma(x)=1\). This quantity is independent of the orthonormal basis
chosen for \(V\).

Small values of \(\cos\gamma(x)\) correspond to point evaluations that are
nearly orthogonal, in coefficient space, to the functional
\(\mathcal L\). Such points amplify relative weight error, so we assume control on the probability that $\cos \gamma(x)$ is vanishingly small.

\begin{definition}[Relative admissibility]
\label{ass:relative-admissibility}
The tuple \((\mu,\rho,V,\mathcal L)\) satisfies the
\emph{relative admissibility condition} if \(Q<\infty\) and
\begin{equation}
\label{eq:relative-admissibility-small-ball}
    \lim_{\epsilon\to0^+}
    \frac{
        \Prob\left[
            \cos\gamma(X)\leq\epsilon
            \,\middle|\,
            X\sim\rho
        \right]
    }{\epsilon^2}
    =
    0.
\end{equation}
Equivalently, with
\begin{equation}
\label{eq:relative-amplification-tail}
    p_J
    \coloneqq
    \Prob\left[
        \frac{1}{\cos\gamma(X)} \ge J^{1/2}
        \,\middle|\,
        X\sim\rho
    \right],
\end{equation}
relative admissibility is $Q < \infty$ along with the condition
\begin{equation}
\label{eq:relative-admissibility}
    \lim_{J\to\infty}Jp_J=0.
\end{equation}
\end{definition}
The small-ball condition \eqref{eq:relative-admissibility-small-ball} is a useful geometric interpretation, while the \(J\)-parametrization will be useful in \Cref{sec:posls}.

\begin{remark}[Absolute-error variant]
\label{rem:absolute-admissibility}
Define the absolute amplification factor
\begin{equation}
\label{eq:absolute-amplification}
    A_{\mathcal L, \rho}(x)
    \coloneqq
    \|\mathcal L\|_{V^*}
    \frac{\sqrt{k(x)}}{\tau^2(x)}.
\end{equation}
For \(J\geq1\), let
\[
    q_J
    \coloneqq
    \Prob\left[
        A_{\mathcal L, \rho}(X)\geq J^{1/2}
        \,\middle|\,
        X\sim\rho
    \right].
\]
Paired with $Q < \infty$, we call
\begin{equation}
\label{eq:absolute-admissibility}
    \lim_{J\to\infty}Jq_J=0
\end{equation}
the \emph{absolute admissibility condition}.  
This condition yields absolute
control of \eqref{eq:weight-approximation-error}; see
\Cref{sec:posls}. 

\end{remark}

\subsubsection{Reference-weight concentration}
\label{ssec:reference-weight-concentration}

Relative and absolute admissibility control approximation of least squares weights relative to the limiting
weights. For example, on any event for which the relative weight error is at most
\(\sigma>0\), i.e., $|\omega_{\mathcal L, \rho}(x_m) - M w_m| \leq \sigma |\omega_{\mathcal L,\rho}(x_m)|$, then
\begin{equation}
\label{eq:condition-number-reference-weight-bound}
    \kappa_{\mathsf W}(\mathcal Q_M)
    \leq
    (1+\sigma)\frac{1}{M}
    \sum_{m=1}^M
    \mathsf K_{\mathcal L,\rho,\mathsf W}(x_m),
\end{equation}
where
\begin{equation}
\label{eq:weighted-reference-magnitude}
    \mathsf K_{\mathcal L,\rho,\mathsf W}(x)
    \coloneqq
    \frac{|\omega_{\mathcal L,\rho}(x)|}{\mathsf W(x)}
    =
    \frac{|L(x)|}{\tau^2(x)\mathsf W(x)}.
\end{equation}

Hence, control of the quadrature condition number can be established by additionally 
requiring concentration of \(\mathsf K_{\mathcal L,\rho,\mathsf W}(X)\), \(X\sim\rho\), around its mean.
This mean is independent of the choice of sampling measure since
\begin{equation}
\label{eq:reference-weight-mean}
    \E_\rho\left[
        \mathsf K_{\mathcal L,\rho,\mathsf W}(X)
    \right]
    =
    \int_D
    \frac{|L(x)|}{\mathsf W(x)}
    \,\dx\mu(x)
    =
    \left\|
        \frac{L}{\mathsf W}
    \right\|_{L^1_\mu(D)}.
\end{equation}
However, the choice of \(\rho\) affects the fluctuations of the empirical mean. These fluctuations are controlled under the following standard condition from high-dimensional probability.

\begin{definition}[Reference-weight concentration]
\label{ass:reference-weight-concentration}
For a fixed approximation weight \(\mathsf W\), the tuple
\((\mu,\rho,V,\mathcal L)\) satisfies the
\emph{reference-weight concentration condition} if
\(\mathsf K_{\mathcal L,\rho,\mathsf W}(X)\) is sub-exponential, i.e.,
\begin{equation}
\label{eq:reference-weight-concentration}
    \left\|
        \mathsf K_{\mathcal L,\rho,\mathsf W}(X)
    \right\|_{\psi_1}
    \coloneqq
    \inf\left\{
        c>0:
        \E_\rho\left[
            \exp\left(
                \frac{
                    \mathsf K_{\mathcal L,\rho,\mathsf W}(X)
                }{c}
            \right)
        \right]
        \leq2
    \right\}
    <\infty,
    \qquad X\sim\rho.
\end{equation}
\end{definition}

\subsection{Main result: Stability of intermediate least squares quadrature}
Our first main result stipulates that, for any prescribed \(0<\xi<1\), a
randomized least squares quadrature rule of a particular finite size yields
weights that concentrate close to their limiting values with probability at
least \(1-\xi\).

\begin{restatable}[Finite-sample weight approximation]{theorem}{finiteSampleWeightApprox}
\label{thm:finite-sample-weight-approximation} 
Let
\(
    \star\in\{\mathrm{rel},\mathrm{abs}\},
\)
and suppose that
\((\mu,\rho,V,\mathcal L)\) is \(\star\)-admissible. Define
\begin{equation}\label{eq:weight_amplification_class}
    \Xi_\star(x)
    \coloneqq
    \begin{cases}
        |\omega_{\mathcal L,\rho}(x)|,
        & \star=\mathrm{rel},\\
        1,
        & \star=\mathrm{abs}.
    \end{cases}
\end{equation}
Then, for every \(\sigma>0\) and every $0<\xi<1$, there exist \(J\geq1\) and an integer \(M\)
satisfying
\[
    \frac{6QJ}{\sigma^2}\log\left(\frac{4N}{\xi}\right)
    \leq
    M
    \leq
    \frac{7QJ}{\sigma^2}\log\left(\frac{4N}{\xi}\right),
\]
such that, with probability at least $1-\xi$, the least squares quadrature rule is
well-defined, exact on \(V\), and satisfies
\[
    \left|
        M w_m-\omega_{\mathcal L,\rho}(x_m)
    \right|
    \leq
    \sigma\,\Xi_\star(x_m),
    \qquad
    m=1,\ldots,M.
\]
\end{restatable}
This conclusion gives a finite-sample guarantee with prescribed success
probability \(1-\xi\). In particular, it guarantees that at least one finite
sample set produces the desired \(M\)-point rule. See
\Cref{subsec:finite-sample-weight-approximation} for the proof and further
details.
Note that the result above restricts $M$ to a particular valid interval, and in
particular one cannot simply take $M$ arbitrarily large: Larger $M$ does
provide more probabilistic concentration, but it also generates more random
weights that must all simultaneously be controlled. 

Relative admissibility alone does not prescribe a rate at which the threshold
\(J\), which drives the sample complexity, may be chosen. (In particular, the
unspecified number $J$ in \Cref{thm:finite-sample-weight-approximation} is
chosen according to the somewhat opaque condition
\eqref{eq:weight-window-choice}.) A more transparent quantitative rate follows
from stronger integrability of the factor \eqref{eq:alignment-angle}. 
\begin{restatable}[Quantitative sample complexity]{lemma}{quantSampleComplexity}
\label{lem:quant-sample-complexity}
Suppose that, for some \(r>2\),
\[
    C_r
    \coloneqq
    \|(\cos\gamma)^{-1}\|_{L^r_\rho}^r
    <
    \infty.
\]
Then the sample size in the $\star = \mathrm{rel}$ case of
\Cref{thm:finite-sample-weight-approximation} may be chosen so that
\[
    M
    =
    O\left(
        \left(\frac{C_r}{\xi}\right)^{2/(r-2)}
        \left(
            \frac{Q}{\sigma^2}\log\left(\frac{4N}{\xi}\right)
        \right)^{r/(r-2)}
    \right).
\]
\end{restatable}
\noindent The proof is given in \Cref{app:quant-sample-complexity}.

Under the additional assumption of reference weight concentration, we obtain a
stable quadrature rule with prescribed probability at least \(1-\xi\), and
hence in particular existence in the sense of
\Cref{sssec:preliminaries-quadrature}.

\begin{restatable}[Stable intermediate least squares rule]{theorem}{stableIntermediateRule}
\label{thm:stable-intermediate-rule}
Fix an approximation weight \(\mathsf W\). Suppose that
\((\mu,\rho,V,\mathcal L)\) is relatively admissible and satisfies the
reference-weight concentration condition
\eqref{eq:reference-weight-concentration}. Then, for every \(\sigma>0\) and every $0<\xi<1$,
there exist \(J\geq1\) and an integer \(M\) satisfying
\begin{equation}
\label{eq:stable-M-window}
    \frac{6QJ}{\sigma^2}\log\left(\frac{6N}{\xi}\right)
    \leq
    M
    \leq
    \frac{7QJ}{\sigma^2}\log\left(\frac{6N}{\xi}\right),
\end{equation}
such that, with probability at least $1-\xi$, the least squares quadrature rule is
well-defined, exact on \(V\), and satisfies
\[
    \kappa_{\mathsf W}(\mathcal Q_M)
    \leq
    (1+\sigma)
    \left[
        \left\|
            \frac{L}{\mathsf W}
        \right\|_{L^1_\mu(D)}
        +
        C
        \left\|
            \mathsf K_{\mathcal L,\rho,\mathsf W}(X)
        \right\|_{\psi_1}
        \frac{\log(6/\xi)}{\sqrt{M}}
    \right],
    \qquad X\sim\rho,
\]
where \(C>0\) is a universal constant.
\end{restatable}
\noindent The proof is given in \Cref{subsec:intermediate-condition-numbers}; the additional requirement needed under absolute weight control is discussed in Remark \ref{rem:absolute-condition-number-reduction}.

These are our most general stability results: the least squares weights are close
to their limiting values, while the induced quadrature rule has controlled
weighted condition number. With additional structure, the same weight
approximation result also yields positive quadrature rules.

\subsection{Quadrature with real and positive weights}
We define some additional properties of $(V,\mathcal{L})$ that enable us to
understand when we achieve quadrature rules with positive weights.
We say that 
\(V\) is closed under complex conjugation if,
\[
    v\in V
    \qquad\Longrightarrow\qquad
    \overline v\in V .
\]
Define the real part of \(V\) by
\begin{equation}
\label{eq:VR-def}
    V_{\mathbb R}
    =
    \{v\in V: v=\overline v\}.
\end{equation}
Then \(V_{\mathbb R}\) is a real vector space and, if $V$ is closed under complex conjugation,
\[
    v
    =
    \frac{v+\overline v}{2}
    +
    i\,\frac{v-\overline v}{2i} \hskip 10pt \Longrightarrow \hskip 10pt
    V = V_{\mathbb R}\oplus iV_{\mathbb R}.
\]
Thus, if \(N=\dim_{\mathbb C}V\) and $V$ is closed under complex conjugation, then
\[
    \dim_{\mathbb R}V_{\mathbb R}=N,
    \qquad
    \dim_{\mathbb R}V=2N.
\]
In particular, any \(V\) closed under conjugation admits an orthonormal basis
consisting of $N$ real-valued functions.

\begin{definition}[Real-preserving functional]
\label{def:real-preserving}
Assume that \(V\) is closed under complex conjugation.  We say that
\(\mathcal L\in V^*\) is \emph{real-preserving} if
\[
    \mathcal L(v)\in\mathbb R
    \qquad
    \text{for every }v\in V_{\mathbb R}.
\]
\end{definition}

The following observation relates this condition to the Riesz
representer of \(\mathcal L\).

\begin{restatable}{lemma}{realPreservingRepresenter}
\label{lem:real-preserving-representer}
Assume that \(V\) is closed under complex conjugation.  Then
\(\mathcal L\) is real-preserving if and only if its \(L^2_\mu\)-Riesz
representer \(L\in V\) is real-valued, up to modification on a
\(\mu\)-null set.
\end{restatable}
\noindent The proof is given in \Cref{app:real-preserving-representer}. Armed with this result, we can articulate the condition needed to discuss positivity of quadrature rules.
\begin{definition}[Nonnegative pair]
\label{def:strictly-positive-pair}
We say that the pair \((V,\mathcal L)\) is \emph{nonnegative} if
\begin{enumerate}
    \item \(V\) is closed under complex conjugation;
    \item \(\mathcal L\) is real-preserving on \(V\); and
    \item the Riesz representer \(L\) of \(\mathcal L\) admits a pointwise
    representative satisfying
    \[
        L(x)\geq 0,
        \qquad x\in D.
    \]
\end{enumerate}
\end{definition}

If $(V,\mathcal{L})$ is nonnegative as prescribed in
Definition~\ref{def:strictly-positive-pair}, one may work with a real-valued
orthonormal basis of \(V_{\mathbb R}\).  In that basis, the moment vector
$\bs{\ell}$ and the least squares weights are real. The large-$M$ limiting
least squares weights are given by, $\frac{L(x_m)}{\tau^2(x_m)}$,
which are strictly nonnegative.  Consequently, if the finite-sample weights
are relatively close to their limiting values with relative error
\(\sigma<1\), then the finite-sample least squares quadrature rule has
nonnegative weights, which are strictly positive at nodes where \(L\) is
positive.

\subsection{Main results: Positive quadrature}
\label{ssec:preliminaries-main}

\begin{theorem}[Positive intermediate least squares rule]
\label{thm:finite-sample-weight-approximation-positive}
  Assume the setup of \Cref{thm:finite-sample-weight-approximation}, and in addition suppose
  that \((V,\mathcal{L})\) is a nonnegative pair. Then, for every
  \(\sigma\in(0,1)\) and \(0<\xi<1\), there exist \(J\geq1\) and an
  integer \(M\) satisfying
  \[
    \frac{6QJ}{\sigma^2}\log\left(\frac{4N}{\xi}\right)
    \leq M \leq
    \frac{7QJ}{\sigma^2}\log\left(\frac{4N}{\xi}\right)
  \]
  such that, with probability at least \(1-\xi\), the least squares
  quadrature rule is well-defined, exact on \(V\), and has real, nonnegative
  weights. If \(L(x_m)>0\) at every sampled node, then all weights are
  strictly positive.
\end{theorem}

The condition of positive weights allows us to state a constructive, general Tchakaloff theorem.
\begin{theorem}[Constructive Generalized Tchakaloff]\label{thm:l2-tchakaloff}
  Assume $(\mu, \rho, V, \mathcal{L})$ satisfy the relative admissibility condition and that $(V,\mathcal{L})$ is a nonnegative pair. Then there is a quadrature rule with at most $N$ nodes in $D$ and positive weights that is exact on $V$.
\end{theorem}
The proof, a direct combination of \Cref{thm:finite-sample-weight-approximation-positive} and
the well-known Carath\'eodory--Steinitz pruning procedure, is given in
\Cref{subsec:reduced}. The result above is ``constructive'' because its proof
is explicitly algorithmic: randomized trials to compute least squares weights,
upon success, can be pruned (through Carath\'eodory--Steinitz-type linear algebra)
to achieve the conclusion. \Cref{thm:l2-tchakaloff} requires $V$ to be closed under conjugation; when $V$ is not closed under conjugation, a Tchakaloff-like result holds on the larger subspace that is the conjugate closure of $V$.
\begin{theorem}[Generalized Tchakaloff]\label{thm:l2-tchakaloff-conjugate}
  Let $(\mu, \rho, V, \mathcal{L})$ be given, and let
  \(\widetilde V = \mathrm{span}\{V, \overline{V}\}\) be the conjugate closure
  of \(V\). Let \(\widetilde{\mathcal L}\in\widetilde V^*\) satisfy
  \(\widetilde{\mathcal L}|_V=\mathcal L\). Assume that
  \((\mu,\rho,\widetilde V,\widetilde{\mathcal L})\) is relatively admissible
  and that \((\widetilde V,\widetilde{\mathcal L})\) is a nonnegative pair.
  Then there is a quadrature rule with at most
  \(P=\dim_{\mathbb C}\widetilde V\leq2N\) nodes in \(D\), with positive
  weights, that is exact on \(\widetilde V\), and hence on \(V\).
\end{theorem}

\subsubsection{Relation to previous work}

Theorem~\ref{thm:l2-tchakaloff} should be viewed as a constructive
\(L^2_\mu\)-based result rather than as the most general abstract existence
theorem.  Schäfer and Ullrich~\cite{schafer_2025} recall Bisgaard's general
Tchakaloff theorem and extend it to complex-valued finite-dimensional
subspaces of \(L^1\), with node count determined by the effective real
dimension.  Our assumptions are stronger: we work in an \(L^2_\mu\) setting
and impose sampling and admissibility hypotheses tailored to weighted least
squares.  The gain is that the proof provides a randomized construction.

\subsection{Discussion: Relative/absolute admissibility and reference-weight concentration} \label{ssec:assumptions-discussion}

We provide some discussion below to interpret our formally abstract conditions of admissibility and reference weight concentration. For brevity, we focus on investigating \textit{relative} admissibility. 

A convenient sufficient condition for relative admissibility, Definition~\ref{ass:relative-admissibility}, is that \( \sec \gamma \propto \sqrt{k}/|L| 
    \in L^2_\rho(D) \) since under that condition,
\begin{equation}\label{eq:relative-admissibility-sufficient}
    Jp_J = \E_\rho \left[J \mathbf 1_{\{ (\cos \gamma (X))^{-1} \ge \sqrt J \}}\right] \le  
    \E_\rho\left[
        (\cos \gamma (X))^{-2}
        \mathbf 1_{\{ (\cos \gamma (X))^{-1} \ge \sqrt J \}}
    \right]
    \longrightarrow 0.
\end{equation}
Similarly, a convenient sufficient condition for absolute admissibility is
\(
    A_{\mathcal L, \rho}\in L^2_\rho(D).
\)
The sufficient condition $\sqrt{k}/|L| \in L^2_\rho$ motivates the observation that relative admissibility is in jeopardy only if $K = \infty$  or if $L_{\inf} = 0$, with the latter defined as
\begin{equation}
\label{eq:Linf}
    L_{\inf}
    =
    \inf_{x\in D}
        \frac{|L(x)|}{\|\mathcal L\|_{V^*}}
    =
    \inf_{x\in D}
        \frac{|\bs\ell^*\bs\varphi(x)|}{\|\bs\ell\|_2}.
\end{equation}
We provide below some examples demonstrating when the relative admissibility condition can be achieved when $K = \infty$ and/or $L_{\inf} = 0$. In all these examples, we assume $Q < \infty$.
\begin{enumerate}
  \item $K < \infty$ and $L_{\inf} > 0$: then the relative admissibility condition is satisfied.
  \item $K < \infty$ and $L_{\inf} = 0$: If $L$ has only isolated zeros, then relative admissibility holds so long as the probability of sampling around zeros of $L$ is not too large. In particular, if $D \subset \R^d$, $\rho$ has a bounded Lebesgue density, and $L$ has only a finite number of isolated zeros having algebraic order at most $\alpha > 0$, then relative admissibility holds if $\alpha < d/2$.
  \item $K = \infty$ and $L_{\inf} > 0$: Choose $\tau^2 = \tau^2_\mathrm{ind} =k/N$.
    \begin{itemize}
      \item Then relative admissibility is implied by the Christoffel tail condition
        \[
            \lim_{J\to\infty}
            J\int_{\{k>J\}} k(x)\,d\mu(x)
            =
            0,
        \]
      \item Let $\mu$ be Lebesgue measure on a compact $D \subset \R^d$, and suppose that all functions in $V$ have no singularity stronger than $\|\x - \mathbf{c}\|_2^{-\alpha}$ for any $\mathbf{c} \in \mathrm{int}(D)$ and $\alpha < d/4$. Then relative admissibility holds.
      \item Let $\mu$ on $\R^d$ satisfy $d \mu(\x) \leq C \exp(-\|\x\|_2^\alpha) d \x$ for sufficiently large $\|\x\|_2$ with any $C, \alpha > 0$. Let $V$ be any finite-dimensional polynomial subspace. Then relative admissibility holds.
    \end{itemize}
  \item If \(\rho\bigl(\{x\in D:L(x)=0 \textrm{ and } k(x)>0\}\bigr)>0\), then relative admissibility is not achievable.
\end{enumerate}
The proofs of these statements are provided in \Cref{app:subspaces}. These situations could be combined to generate relatively admissible tuples with both $K = \infty$ and $L_{\inf} = 0$.

For any fixed subspace $V$, our asymptotic characterizations of relative and absolute admissibility, i.e., $J p_J \rightarrow 0$ and $J q_J \rightarrow 0$, respectively, are stronger than necessary. The concentration proofs in \Cref{sec:posls} require these conditions only so that one can pick \textit{some} $J \geq 1$ sufficiently large satisfying $r_J \leq 1/2$ and
\begin{align*}
  J r_J \lesssim
  \frac{\sigma^2 \xi}{Q \log(N/\xi)},
\end{align*}
for $r_J \in \{p_J,q_J\}$.
Hence, the weaker conditions above are sufficient to replace the roles of relative and absolute admissibility, respectively. In principle these conditions also reveal how $J$ should be chosen. For example, if it is known \textit{a priori} that there is some $\beta > 1$ so that $r_J = \mathcal{O}(1/J^\beta)$ for large $J$, then choosing
\begin{align*}
  J \gtrsim \left( \frac{Q \log(N/\xi)}{\sigma^2 \xi} \right)^{1/(\beta-1)},
\end{align*}
is sufficient for the purposes of our theoretical results.

Interpretations for reference weight concentration, i.e., that if $X \sim \rho$ we have \(\mathsf K_{\mathcal L,\rho,\mathsf W}(X)
    =
    \frac{|L(X)|}{\tau^2(X)\mathsf W(X)}\) is subexponential, are more established. We refer to, e.g., \cite[Section 2.7]{vershynin2019high}.

\subsubsection{Mixture sampling}
\label{sssec:mixture-sampling}

In many cases we can provide a precise recommendation for achieving both
relative admissibility and reference-weight concentration. Since
\(\dx\rho=\tau^2\dx\mu\),
\[
    \mathsf K_{\mathcal L,\rho,\mathsf W}(x)\,\dx\rho(x)
    =
    \frac{|L(x)|}{\mathsf W(x)}\,\dx\mu(x).
\]
Thus, reference-weight concentration prevents \(\rho\) from severely
undersampling regions carrying substantial
\((|L|/\mathsf W)\,\dx\mu\)-mass. This suggests balancing the induced and
weighted reference measures. Provided
\(L/\mathsf W\in L^1_\mu(D)\), for \(0<\theta<1\) define
\begin{equation}
\label{eq:mixed-sampling-density}
    \tau_\theta^2(x)
    =
    \theta\frac{k(x)}{N}
    +
    (1-\theta)
    \frac{|L(x)|/\mathsf W(x)}
    {\|L/\mathsf W\|_{L^1_\mu(D)}}.
\end{equation}
Then the sampling measure
\(\dx\rho_\theta=\tau_\theta^2\dx\mu\) satisfies
\[
    Q\leq\frac{N}{\theta},
    \qquad
    \mathsf K_{\mathcal L,\rho_\theta,\mathsf W}(x)
    \leq
    \frac{\|L/\mathsf W\|_{L^1_\mu(D)}}{1-\theta}.
\]
Hence this choice retains near-optimal Gramian control while making the
weighted reference magnitude uniformly bounded.

\begin{restatable}[Mixture sampling]{lemma}{mixedSamplingRelativeAdmissibility}
\label{lemma:mixed-sampling-relative-admissibility}
Fix an approximation weight \(\mathsf W\) satisfying
\[
    \frac{L}{\mathsf W} \in L^1_\mu(D),\qquad\text{and,}\qquad
    \mathsf W^{-1}\in L^2_\mu(D),
\]
and assume that
\((\mu,\rho_{\mathrm{ind}},V,\mathcal L)\) satisfies the relative
admissibility condition, where the induced sampling measure \(\dx\rho_{\mathrm{ind}}(x)\) is as in \eqref{eq:induced_measure}. Then, for every \(\theta\in(0,1)\), the tuple
\((\mu,\rho_\theta,V,\mathcal L)\), with \(\rho_\theta\) defined by
\eqref{eq:mixed-sampling-density}, satisfies both relative admissibility and
reference-weight concentration with respect to \(\mathsf W\).
\end{restatable}
\noindent The proof is given in \Cref{app:mixed-sampling-relative-admissibility}.

\begin{remark}
When \(\mathsf W\equiv1\), the condition
\(\mathsf W^{-1}\in L^2_\mu(D)\) reduces to $D$ having finite $\mu$-measure. 
Another weighting commonly considered in the case of exponential approximation weights is
\[
    \mathsf W^2(x)
    =
    \frac{\dx\mu}{\dx x}(x)
\implies
    \|\mathsf W^{-1}\|_{L^2_\mu(D)}^2
    =
    \int_D 1\,\dx x,
\]
which is infinite whenever \(D\) has infinite Lebesgue measure. 
We note that these cases can be overly restrictive on domains with infinite measure. 
In such cases, relative admissibility of the
mixture may instead be established directly from the corresponding weighted
tail condition.
One final weighting that is immediate from the construction of the sampling measure is
\[
    \mathsf W^{-2}(x)
    =
    \tau^2(x)
\implies
    \|\mathsf W^{-1}\|_{L^2_\mu(D)}^2
    =
    \int_D \dx\rho(x) = 1.
\]
\end{remark}

\subsubsection{Distinctions between the conditions}
\label{sssec:admissibility-comparison}
The three conditions we have introduced (relative admissibility, absolute admissibility, and reference weight concentration) control different probabilistic tails with $X \sim \rho$:
\begin{itemize}
  \item relative admissibility controls small values of \(\cos\gamma(X) \propto L(X)/\sqrt{k(X)}\),
  \item absolute admissibility controls large values of \(A_{\mathcal L, \rho}(X) \propto \sqrt{k(X)}/\tau^2(X)\), and
  \item reference-weight concentration controls large values of
  \[
      \mathsf K_{\mathcal L,\rho,\mathsf W}(X)
      =
      \frac{
          A_{\mathcal L,\rho}(X)\cos\gamma(X)
      }{\mathsf W(X)}
      =
      \frac{|L(X)|}
      {\tau^2(X)\mathsf W(X)}.
  \]
\end{itemize}
In particular, relative and absolute admissibility need not imply
reference-weight concentration.

For example, take \(\mathsf W\equiv1\), let
\[
    V=\operatorname{span}\{\varphi\},
    \qquad
    \|\varphi\|=1,
    \qquad
    L=\varphi,
\]
and use induced sampling
\(
    \dx\rho=|\varphi|^2\,\dx\mu.
\)
Then \(Q=1\), \(\cos\gamma=1\), and
\[
A_{\mathcal L,\rho}
=
\mathsf K_{\mathcal L,\rho,1}
=
|\omega_{\mathcal L,\rho}|
=
\frac1{|\varphi|}.
\]
Suppose that \(D\subset\mathbb R^d\), that \(\mu\) is locally comparable
to Lebesgue measure, and that \(\varphi\) has an isolated algebraic zero
\(x_0\) of order \(\beta>0\):
\[
    |\varphi(x)|
    \asymp
    \|x-x_0\|^\beta
    \qquad
    \text{and}
    \qquad
    \dx \rho \asymp \| x-x_0 \|^{2\beta} \dx x
\]
near \(x_0\). Then, as \(t\to\infty\),
\[
    \Prob_\rho\left[
        |\omega_{\mathcal L, \rho}(X)|\geq t
    \right]
    \asymp
    \int_{D}\mathbf 1_{\{\| x - x_0 \|^\beta \le t^{-1}\}} \| x - x_0 \|^{2\beta} \dx x
    \asymp
    \int_0^{t^{-1/ \beta }} r^{2\beta} r^{d-1}\dx r
    \asymp
    t^{-(2+d/\beta)}.
\]
The tail exponent \(2+d/\beta\) is strictly greater than two, and therefore
\[
    J\Prob_\rho\left[
        A_{\mathcal L, \rho}(X)\geq J^{1/2}
    \right]
    \asymp
    J^{-d/(2\beta)}
    \longrightarrow0.
\]
Thus, relative admissibility holds because \(\cos\gamma=1\), and absolute
admissibility holds because the tail has order greater than two. However,
 \(\mathsf K_{\mathcal L,\rho,1}(X)\) has only polynomial tails and is not
subexponential, so reference-weight concentration fails.

\subsection{Comparison to results in \cite{migliorati2022stable,schafer_2025}}

The actual computational procedure we propose is identical to the basic method
in \cite{migliorati2022stable}. (Although \cite{migliorati2022stable} proposes
additional adaptive variants.) However, our procedure allows $\mu$ to be an
infinite (still positive) measure, $V$ can contain complex-valued functions,
and the sampler $\tau^2$ and approximation weight $\mathsf W$ can be chosen
flexibly. In contrast, \cite{migliorati2022stable} considers the real-valued
function setting, assumes constants lie in $V$, and limits to $\mathcal{L}$ the
integral operator, $\tau^2 = k/N$, and $\mathsf W^{-2} \propto k$. When considering
positive quadrature and weights, our setup specialized to $K < \infty$ and
$L_{\inf} > 0$ coincides with the assumptions and sample complexity required in,
e.g., \cite[Theorem 6]{migliorati2022stable}.

Nevertheless, the force of our theoretical results
are incomparable to those in \cite{migliorati2022stable}; the latter
investigates probabilistic expectation of quadrature rule errors and contains
some discussion specialized to polynomial subspaces $V$. Our results instead
give prescribed-probability guarantees for the stability of individual sampled
rules, although the choice of \(J\) under relative admissibility may remain
implicit in more general scenarios.

Compared to \cite{schafer_2025} we provide an explicit constructive
procedure through randomness, but we require some comparatively mild
assumptions (relative admissibility) in order to guarantee that our procedures
succeed.

\section{Least squares approximations}
\label{sec:ls}

We now construct the intermediate \(M\)-point quadrature rule used in
\Cref{thm:finite-sample-weight-approximation}.  The rule is obtained by sampling
nodes \(x_1,\ldots,x_M\sim\rho\), forming a weighted least squares
approximation in \(V\), and applying \(\mathcal L\) to this approximation.
When the empirical Gramian is invertible, the resulting rule is exact on
\(V\).  This section derives the rule and identifies the limiting form of
its weights; finite-sample estimates are proved in
Section~\ref{sec:posls}.

\subsection{Weighted least squares}
Fix the sampling measure \(\dx\rho=\tau^2\dx\mu\) introduced in
\Cref{ssec:notation-sampling}, and let
\(x_1,\ldots,x_M\) be independent samples from \(\rho\). Since
\(\rho(\{\tau^2=0\})=0\), the factors \(1/\tau^2(x_m)\) are finite almost
surely. The factor \(1/\tau^2\) corrects for sampling from \(\rho\) while forming
\(L^2_\mu\) inner products.

Given \(u\in\mathcal U(D)\) and an orthonormal basis
\(\{\varphi_n\}_{n=1}^N\) for \(V\), define the weighted discrete least
squares approximant \(u_V^{(M)}\in V\) by
\begin{align}
    \label{eq:wls-functional}
    u_V^{(M)}
    =
    \argmin_{v\in V}
    \frac{1}{M}\sum_{m=1}^M
    \frac{|v(x_m)-u(x_m)|^2}{\tau^2(x_m)}.
\end{align}
Here and below, sampled values are taken for fixed pointwise-defined
representatives, as in Section~\ref{sec:preliminaries}. Writing
\[
    u_V^{(M)}=\sum_{n=1}^N c_n\varphi_n,
\]
the coefficient vector $\mathbf c=\trans{(c_1,\ldots,c_N)}\in\mathbb C^N$ satisfies the normal equations
\begin{align}
    \label{eq:normal-equations}
    \mathbf G^{(M)}\mathbf c
    =
    \mathbf V^*\mathbf W\mathbf u,
    \qquad
    \mathbf G^{(M)}
    =
    \mathbf V^*\mathbf W\mathbf V,
\end{align}
where
\begin{align}
    \label{eq:V-W-def}
    (\mathbf V)_{m,n}=\varphi_n(x_m),
    \qquad
    (\mathbf W)_{m,m}=\frac{1}{M\tau^2(x_m)},
    \qquad
    \mathbf u=\trans{(u(x_1),\ldots,u(x_M))}.
\end{align}
If $u\in V$ and $\mathbf G^{(M)}$ is invertible, then $u_V^{(M)}=u$. Hence, applying \(\mathcal L\) to \(u_V^{(M)}\)
defines a quadrature rule exact on \(V\).

\subsection{Empirical Gramian}

The empirical Gramian in \eqref{eq:normal-equations} has entries
\begin{align}
    \label{eq:GM-def}
    (\mathbf G^{(M)})_{j,k}
    =
    \frac{1}{M}\sum_{m=1}^M
    \frac{\overline{\varphi_j(x_m)}\varphi_k(x_m)}
    {\tau^2(x_m)}.
\end{align}
Consequently, by the coverage condition \eqref{eq:sampling-coverage},
\begin{equation}\label{eq:G-asymptotic}
    \mathbb E_\rho[\mathbf G^{(M)}]=\mathbf I,
    \qquad
    \text{and}
    \qquad
    \mathbf G^{(M)}\longrightarrow\mathbf I
    \quad\text{a.s}.
\end{equation}
Finite-sample concentration is governed by the sampling parameter \(Q\) defined in \eqref{eq:Q-def} and is developed in \Cref{sec:posls}.

\subsection{Least squares weights and their limit}\label{ssec:ls-weights}

Recall the moment vector \(\bs\ell\), Riesz representer \(L\), and
reference-weight function
\(
    \omega_{\mathcal L, \rho}(x)
    =\overline{L(x)}\tau^{-2}(x)
\)
from \Cref{sssec:preliminaries-L,ssec:weight-approximation}.

For
\(u_V^{(M)}=\sum_{n=1}^N c_n\varphi_n\), we have 
\[
    \mathcal L(u_V^{(M)})
    =
    \trans{\bs\ell}\mathbf c.
\]

If $\mathbf G^{(M)}$ is invertible, then substituting the solution of \eqref{eq:normal-equations} into the above gives
\begin{align}
    \label{eq:weights-explicit}
    \mathcal L\left(u_V^{(M)}\right)
    =
    \trans{\bs\ell}
    \left(\mathbf G^{(M)}\right)^{-1}
    \mathbf V^*\mathbf W\mathbf u
    =
    \trans{\mathbf w}\mathbf u.
\end{align}
Thus, the least squares construction defines an $M$-point quadrature rule with weights $\mathbf w\in\mathbb C^M$:
\begin{align}
    \label{eq:w-vector}
    \trans{\mathbf w}
    =
    \trans{\bs\ell}
    \left(\mathbf G^{(M)}\right)^{-1}
    \mathbf V^*\mathbf W.
\end{align}
The formula \eqref{eq:w-vector} is written in an arbitrary complex
orthonormal basis; however, under the real-structure assumptions used later for
positivity, the resulting weights are real.

\begin{restatable}[Reality of the least squares weights]{lemma}{realLeastSquaresWeights}
\label{lemma:real-ls-weights}
Assume that \(V\) is closed under complex conjugation and that
\(\mathcal L\) is real-preserving.  If \(\mathbf G^{(M)}\) is invertible,
then the least squares weight vector \(\mathbf w\) defined by
\eqref{eq:w-vector} belongs to \(\mathbb R^M\).
\end{restatable}
\noindent The proof is straightforward and is given in \Cref{app:Real-ls-weights}.

We now turn to limiting behavior. Whenever \(\mathbf G^{(M)}\) is invertible,
the individual weight at the node \(x_m\) is
\begin{align}
    \label{eq:weights-individual}
    w_m
    =
    \frac{1}{M\tau^2(x_m)}
    \trans{\bs\ell}
    \left(\mathbf G^{(M)}\right)^{-1}
    \overline{\bs\varphi(x_m)}.
\end{align}
The following lemma records the almost-sure asymptotic behavior of these least squares weights.

\begin{lemma}
\label{lemma:wm-asymptotic}
Let $\{x_j\}_{j=1}^\infty$ be independent samples from $\rho$, and fix
\(m\ge1\). For \(M\ge m\) such that \(\mathbf G^{(M)}\) is invertible, let
\(w_m^{(M)}\) be the weight defined by \eqref{eq:weights-individual} using
the first \(M\) samples. Then, almost surely, \(\mathbf G^{(M)}\) is
invertible for all sufficiently large \(M\), and
\begin{align*}
    \lim_{M\to\infty} M w_m^{(M)}
    =
    \frac{\overline{L(x_m)}}{\tau^2(x_m)} = \omega_{\mathcal L, \rho}(x_m)
    \qquad
    \text{a.s}.
\end{align*}
\end{lemma}

\begin{proof}
By \eqref{eq:G-asymptotic},  $\mathbf G^{(M)}\to\mathbf I$ almost surely. On
this event, \(\mathbf G^{(M)}\) is invertible for all sufficiently large
\(M\), and continuity of matrix inversion gives
\[
    \left(\mathbf G^{(M)}\right)^{-1}\to\mathbf I.
\]
Substituting this limit into \eqref{eq:weights-individual} yields
\[
    M w_m^{(M)}
    =
    \frac{
    \trans{\bs\ell}
    \left(\mathbf G^{(M)}\right)^{-1}
    \overline{\bs\varphi(x_m)}
    }
    {\tau^2(x_m)}
    \to
    \frac{\trans{\bs\ell}\,\overline{\bs\varphi(x_m)}}{\tau^2(x_m)}
    =
    \frac{\overline{L(x_m)}}{\tau^2(x_m)}.
\]
\end{proof}

Lemma~\ref{lemma:wm-asymptotic} shows that the least squares weights inherit
their limiting magnitude, and in the real-preserving case their limiting
sign, from \(\omega_{\mathcal L, \rho}(x)\).  Equivalently, the idealized
sample-dependent rule associated with these limiting weights is, for
\(u\in\mathcal U(D)\),
\begin{align}
    \label{eq:reference-weight-rule}
    \mathcal Q_M^{\mathrm{ref}}(u)
    =
    \frac{1}{M}\sum_{m=1}^M \omega_{\mathcal L, \rho}(x_m)u(x_m).
\end{align}
This rule is unbiased for \(\mathcal L\) on \(V\).  Indeed, for \(v\in V\),
\begin{align*}
    \mathbb E_\rho[\mathcal Q_M^{\mathrm{ref}}(v)]
    &=
    \int_D
    \omega_{\mathcal L, \rho}(x)v(x)
    \dx{\rho}(x)
    =
    \int_D \overline{L(x)}v(x)\dx{\mu}(x)
    =
    \mathcal L(v).
\end{align*}
Moreover, repeated samples carry identical reference weights, so their
coefficients do not cancel, and
\[
    \kappa_{\mathsf W}(\mathcal Q_M^{\mathrm{ref}})
    =
    \frac{1}{M}
    \sum_{m=1}^M
    \frac{
        |\omega_{\mathcal L,\rho}(x_m)|
    }{
        \mathsf W(x_m)
    }
    =
    \frac{1}{M}
    \sum_{m=1}^M
    \mathsf K_{\mathcal L,\rho,\mathsf W}(x_m).
\]
Hence, provided \(L/\mathsf W\in L^1_\mu(D)\),
\[
    \mathbb E_\rho\left[
        \kappa_{\mathsf W}(\mathcal Q_M^{\mathrm{ref}})
    \right]
    =
    \left\|
        \frac{L}{\mathsf W}
    \right\|_{L^1_\mu(D)}.
\]
Thus, the mean weighted reference magnitude is independent of the sampling
measure, while its fluctuations and the finite-sample Gramian depend on
\(\rho\).

The finite-sample rule \eqref{eq:weights-individual} differs from \(\mathcal Q_M^{\mathrm{ref}}\) through the deviation of the empirical Gramian from the identity. In \Cref{sec:posls}, we first derive a deterministic bound for \[ \left| M w_m-\omega_{\mathcal L, \rho}(x_m) \right| \] in terms of \(\|(\mathbf G^{(M)})^{-1}-\mathbf I\|_2\), and then combine Gramian concentration with the relative and absolute admissibility conditions of \Cref{ssec:weight-approximation}. The reference-weight concentration condition of \Cref{ssec:reference-weight-concentration} is subsequently used to control the weighted condition number of the resulting rule through the weighted reference magnitude \(\mathsf K_{\mathcal L,\rho,\mathsf W}\).

\section{Concentration results}
\label{sec:posls}

Section~\ref{sec:ls} showed that, whenever the empirical Gramian \(\mbf G^{(M)}\) is invertible, the least squares weights satisfy 
\[ 
    M w_m = \frac{ \trans{\bs\ell} \left(\mbf G^{(M)}\right)^{-1} \overline{\bs\varphi(x_m)} }{\tau^2(x_m)}, 
\] 
and converge, for each fixed \(m\), to the reference weight \(\omega_{\mathcal L,\rho}(x_m)\). We now quantify this approximation at finite \(M\). We first control the deviation of \(\left(\mbf G^{(M)}\right)^{-1}\) from the identity. We then combine this estimate with the relative and absolute admissibility conditions of \Cref{ssec:weight-approximation} to obtain uniform finite-sample control of the quadrature weights. Finally, the reference-weight concentration condition of \Cref{ssec:reference-weight-concentration} yields bounds on the weighted condition number of the intermediate rule.

\subsection{Concentration of the empirical Gramian}
\label{subsec:gramian-concentration}

Recall from \eqref{eq:GM-def} that 
\[ 
    \mbf G^{(M)} = \frac{1}{M} \sum_{m=1}^M \frac{ \overline{\bs\varphi(x_m)} \trans{\bs\varphi(x_m)} }{\tau^2(x_m)} 
    \qquad\text{and}\qquad 
    \E_\rho\!\left[\mbf G^{(M)}\right] = \mbf I.
\]
The size of each rank-one summand is controlled by the sampling parameter \(Q\) in \eqref{eq:Q-def}. The following estimate is the standard concentration result for randomized
weighted least squares, whose proof is given in \Cref{app:G-chernoff}; see also \cite{cohen_stability_2013,cohen_optimal_2017}.

\begin{restatable}[Gramian concentration]{lemma}{gramianConcentration}
\label{lemma:G-chernoff}
    Let \(\delta\in(0,1)\) and \(\varepsilon>0\).  Suppose that \(Q<\infty\).
    Then
    \begin{align}
        \label{eq:Ghat-sampling-criterion}
        M
        \ge
        \frac{3Q}{\delta^2}
        \log\left(\frac{2N}{\varepsilon}\right)
        \implies
        \Prob\left[
            \left\|\mbf G^{(M)}-\mbf I\right\|_2>\delta
        \right]
        \leq \varepsilon .
    \end{align}
\end{restatable}

We require the corresponding estimate for the inverse Gramian.
In all probability statements below, we interpret
\(\|(\mbf G^{(M)})^{-1}-\mbf I\|_2=+\infty\) when \(\mbf G^{(M)}\) is
singular. Thus, every event imposing a finite upper bound on this quantity
entails invertibility.

\begin{lemma}[Inverse Gramian concentration]
\label{lemma:F-chernoff}
    Let \(c_0 \coloneqq 1 - 2^{-1/2}\). Suppose that \(Q < \infty\), \(\delta\in(0,c_0]\), and \(\varepsilon>0\). If
    \begin{align}
        \label{eq:Fhat-sampling-criterion}
        M
        \ge
        \frac{6Q}{\delta^2}
        \log\left(\frac{2N}{\varepsilon}\right)
        \implies
        \Prob\left[
            \left\|
            \left(\mbf G^{(M)}\right)^{-1}-\mbf I
            \right\|_2>\delta
        \right]
        \leq \varepsilon .
    \end{align}
\end{lemma}

\begin{proof}
    On the event
    \[
        \left\|\mbf G^{(M)}-\mbf I\right\|_2
        \leq
        \frac{\delta}{\sqrt 2},
    \]
    the Gramian is invertible and
    \[
        \left\|
        \left(\mbf G^{(M)}\right)^{-1}
        \right\|_2
        \leq
        \frac{1}{1-\delta/\sqrt 2}
        \leq
        \sqrt 2.
    \]
    Therefore, on this event,
    \begin{align*}
        \left\|
        \left(\mbf G^{(M)}\right)^{-1}-\mbf I
        \right\|_2
        =
        \left\|
        \left(\mbf G^{(M)}\right)^{-1}
        \left(\mbf I-\mbf G^{(M)}\right)
        \right\|_2
        \leq
        \sqrt 2
        \left\|
        \mbf G^{(M)}-\mbf I
        \right\|_2
        \leq \delta .
    \end{align*}
    Hence,
    \begin{align*}
        \Prob\left[
            \left\|
            \left(\mbf G^{(M)}\right)^{-1}-\mbf I
            \right\|_2>\delta
        \right]
        &\leq
        \Prob\left[
            \left\|\mbf G^{(M)}-\mbf I\right\|_2
            >
            \frac{\delta}{\sqrt 2}
        \right].
    \end{align*}
    Applying Lemma~\ref{lemma:G-chernoff} with
    \(\delta/\sqrt 2\) in place of \(\delta\) gives the result.
\end{proof}

\subsection{Deterministic absolute and relative weight bounds}
\label{subsec:deterministic-weight-bounds}

We now convert inverse-Gramian concentration into pointwise control of the
least squares weights.  The following deterministic estimate is the basic
ingredient.  It separates the numerical stability condition
\(
    \left\|
    \left(\mbf G^{(M)}\right)^{-1}-\mbf I
    \right\|_2 \leq \delta
\)
from the sampling conditions imposed on the individual nodes. Recall that
\[
    \omega_{\mathcal L,\rho}(x)
    =
    \frac{\overline{L(x)}}{\tau^2(x)},
    \qquad\text{and}\qquad
    A_{\mathcal L,\rho}(x)
    =
    \|\mathcal L\|_{V^*}
    \frac{\sqrt{k(x)}}{\tau^2(x)}.
\]

\begin{lemma}[Deterministic weight perturbation]
\label{lemma:deterministic-weight-perturbation}

Assume that \(\mbf G^{(M)}\) is invertible. Then, for every
\(m=1,\ldots,M\),
\begin{equation}
\label{eq:deterministic-absolute-weight-bound}
    \left|
        M w_m-\omega_{\mathcal L,\rho}(x_m)
    \right|
    \leq
    \left\|
        \left(\mbf G^{(M)}\right)^{-1}-\mbf I
    \right\|_2
    A_{\mathcal L,\rho}(x_m).
\end{equation}
Moreover, whenever \(\cos\gamma(x_m)>0\),
\begin{equation}
\label{eq:deterministic-relative-weight-bound}
    \frac{\left|
        M w_m-\omega_{\mathcal L,\rho}(x_m)
    \right|}{\left|
        \omega_{\mathcal L,\rho}(x_m)
    \right|}
    \leq
    \frac{
        \left\|
            \left(\mbf G^{(M)}\right)^{-1}-\mbf I
        \right\|_2
    }{\cos\gamma(x_m)}
    .
\end{equation}
\end{lemma}

\begin{proof}
By the explicit formula for the least squares weights in \eqref{eq:weights-individual},
\begin{align*}
    M w_m-\omega_{\mathcal L,\rho}(x_m)
    &=
    \frac{
        \trans{\bs\ell}
        \left[
            \left(\mbf G^{(M)}\right)^{-1}-\mbf I
        \right]
        \overline{\bs\varphi(x_m)}
    }{\tau^2(x_m)}.
\end{align*}
Therefore,
\begin{align*}
    \left|
        M w_m-\omega_{\mathcal L,\rho}(x_m)
    \right|
    &\leq
    \frac{
        \|\bs\ell\|_2
        \left\|
            \left(\mbf G^{(M)}\right)^{-1}-\mbf I
        \right\|_2
        \|\bs\varphi(x_m)\|_2
    }{\tau^2(x_m)} 
    =
    \left\|
        \left(\mbf G^{(M)}\right)^{-1}-\mbf I
    \right\|_2
    A_{\mathcal L,\rho}(x_m),
\end{align*}
where we used \(\|\bs\ell\|_2=\|\mathcal L\|_{V^*}\) and \(\|\bs\varphi(x_m)\|_2=\sqrt{k(x_m)}.\)
This proves \eqref{eq:deterministic-absolute-weight-bound}. Finally,
\(
    \left|\omega_{\mathcal L,\rho}(x)\right|
    =
    A_{\mathcal L,\rho}(x)\cos\gamma(x).
\)
When \(\cos\gamma(x_m)>0\), substituting
\[
    A_{\mathcal L,\rho}(x_m)
    =
    \frac{
        |\omega_{\mathcal L,\rho}(x_m)|
    }{\cos\gamma(x_m)}
\]
into \eqref{eq:deterministic-absolute-weight-bound} gives
\eqref{eq:deterministic-relative-weight-bound}.
\end{proof}

The two admissibility conditions in \Cref{ssec:weight-approximation} correspond to controlling the two
amplification factors appearing in
Lemma \ref{lemma:deterministic-weight-perturbation}.

\begin{corollary}[Good-node weight bounds]
\label{cor:deterministic-good-node-bounds}
Let \(J\geq1\), \(\sigma>0\), and suppose that
\[
    \left\|
        \left(\mbf G^{(M)}\right)^{-1}-\mbf I
    \right\|_2
    \leq
    \frac{\sigma}{\sqrt J}.
\]
Then the following statements hold.
\begin{enumerate}
    \item If
    \(
        A_{\mathcal L,\rho}(x_m)<\sqrt J,
    \)
    then
    \(
        \left|
            M w_m-\omega_{\mathcal L,\rho}(x_m)
        \right|
        \leq \sigma.
    \)

    \item If
    \(
        \cos\gamma(x_m)>\frac{1}{\sqrt J},
    \)
    then
    \(
        \left|
            M w_m-\omega_{\mathcal L,\rho}(x_m)
        \right|
        \leq
        \sigma
        \left|
            \omega_{\mathcal L,\rho}(x_m)
        \right|.
    \)
\end{enumerate}
\end{corollary}

\subsection{An admissible sample-size window}
\label{subsec:sample-size-window}

Let \(\mathcal G_J\subseteq D\) be a measurable good-node set indexed by
\(J\geq1\), and write
\[
    r_J
    \coloneqq
    \rho(D\setminus\mathcal G_J).
\]
For example, we can take 
\[
   \left(\mathcal G_J^{\mathrm{rel}}
    =
    \left\{
        x\in D:
        \cos\gamma(x)>\frac1{\sqrt J}
    \right\},
    r_J=p_J\right)
    \qquad 
    \text{and}
    \qquad
    \left(\mathcal G_J^{\mathrm{abs}}
    =
    \left\{
        x\in D:
        A_{\mathcal L,\rho}(x)<\sqrt J
    \right\},
    r_J=q_J\right)
\]
for relative and absolute weight control, respectively.
Generally, for independent samples \(x_1,\ldots,x_M\sim\rho\), define
\[
    R_M(J)
    \coloneqq
    \bigcap_{m=1}^M\{x_m\in\mathcal G_J\},
\]
the event that all nodes are good-nodes. 
Then
\[
    \Prob[R_M(J)]
    =
    (1-r_J)^M.
\]
The following lemma gives a range of sample sizes for which inverse-Gramian
concentration and the good-node event are simultaneously compatible.

\begin{lemma}[Admissible sample-size window]
\label{lemma:admissible-M-window}
Let \(0<\varepsilon,\beta<1\), \(\sigma>0\), and \(J\geq1\). Suppose that
\(r_J\leq1/2\),
\begin{equation}\label{eq:window-compatibility-condition}
    \frac{QJ}{\sigma^2}
    \log\left(\frac{2N}{\varepsilon}\right)
    \geq1,
    \qquad 
    \text{and}
    \qquad
    \frac{14QJr_J}{\sigma^2}
    \log\left(\frac{2N}{\varepsilon}\right)
    <
    \log\left(\frac{1}{1-\beta}\right).
\end{equation}
Then there exists an integer \(M\) satisfying
\begin{equation}\label{eq:window-length-condition}
    \frac{6QJ}{\sigma^2}
    \log\left(\frac{2N}{\varepsilon}\right)
    \leq M
    \leq
    \frac{7QJ}{\sigma^2}
    \log\left(\frac{2N}{\varepsilon}\right)
\end{equation}
such that
\(
    \Prob[R_M(J)]>1-\beta.
\)
If, in addition,
\(
    \frac{\sigma}{\sqrt J}\leq c_0,
\)
then
\[
    \Prob\left[
        \left\|
            \left(\mbf G^{(M)}\right)^{-1}-\mbf I
        \right\|_2
        \leq
        \frac{\sigma}{\sqrt J}
    \right]
    \geq
    1-\varepsilon.
\]
\end{lemma}

\begin{proof}
Set
\[
    B
    \coloneqq
    \frac{QJ}{\sigma^2}
    \log\left(\frac{2N}{\varepsilon}\right).
\]
By assumption, \(B\geq1\), so the interval \([6B,7B]\) contains an integer \(M\). 
Now, since the samples are independent,
\(
    \Prob[R_M(J)]
    =
    (1-r_J)^M.
\)
Moreover, \(r_J\leq1/2\) implies
\(
    -\log(1-r_J)\leq2r_J.
\)
Therefore,
\begin{align*}
    -\log\Prob[R_M(J)]
    &=
    -M\log(1-r_J) 
    \leq
    2Mr_J 
    \leq
    14Br_J 
    <
    \log\left(\frac{1}{1-\beta}\right).
\end{align*}
Exponentiating gives
\(
    \Prob[R_M(J)]>1-\beta.
\)
Finally, suppose that
\(
    \frac{\sigma}{\sqrt J}\leq c_0
\)
and set
\(
    \delta
    =
    \frac{\sigma}{\sqrt J}.
\)
The lower bound on \(M\) gives
\[
    M
    \geq
    \frac{6QJ}{\sigma^2}
    \log\left(\frac{2N}{\varepsilon}\right)
    =
    \frac{6Q}{\delta^2}
    \log\left(\frac{2N}{\varepsilon}\right).
\]
Hence, Lemma \ref{lemma:F-chernoff} yields
\[
    \Prob\left[
        \left\|
            \left(\mbf G^{(M)}\right)^{-1}-\mbf I
        \right\|_2
        \leq
        \frac{\sigma}{\sqrt J}
    \right]
    \geq
    1-\varepsilon.
\]
\end{proof}

For both relative and absolute error, the admissibility condition is precisely
\(
    Jr_J\longrightarrow0.
\)
Thus, for fixed \(\varepsilon,\beta,\sigma\), the compatibility condition
\eqref{eq:window-compatibility-condition} holds for all sufficiently large
\(J\). The remaining conditions
\eqref{eq:window-length-condition} and
\(\sigma/\sqrt J\leq c_0\) also hold once \(J\) is sufficiently large. This is made explicit in the following. 

\subsection{Finite-sample weight approximation}
\label{subsec:finite-sample-weight-approximation}

We now prove \Cref{thm:finite-sample-weight-approximation} which 
specializes the admissible sample-size window to the relative and
absolute admissibility conditions.

\finiteSampleWeightApprox*
\begin{proof} For
\(\star\in\{\mathrm{rel},\mathrm{abs}\}\), define
\[
    \mathcal G_J^\star
    \coloneqq
    \begin{cases}
        \left\{
            x\in D:
            \cos\gamma(x)>J^{-1/2}
        \right\},
        & \star=\mathrm{rel},\\[0.3em]
        \left\{
            x\in D:
            A_{\mathcal L,\rho}(x)<\sqrt J
        \right\},
        & \star=\mathrm{abs},
    \end{cases}
\]
and
\[
    r_J^\star
    \coloneqq
    \rho(D\setminus\mathcal G_J^\star)
    =
    \begin{cases}
        p_J,
        & \star=\mathrm{rel},\\
        q_J,
        & \star=\mathrm{abs}.
    \end{cases}
\]
Under the corresponding admissibility condition,
\[
    Jr_J^\star\longrightarrow0
    \qquad\text{as }J\to\infty.
\]
We may therefore choose
\begin{equation}\label{eq:J-treshold}
    J\geq
    \max\left\{
        1,
        \frac{\sigma^2}{c_0^2},
        \frac{\sigma^2}{Q\log(4N/\xi)}
    \right\}
\end{equation}
sufficiently large that
\[
    r_J^\star\leq\frac12
\]
and
\begin{equation}
\label{eq:weight-window-choice}
    \frac{14QJr_J^\star}{\sigma^2}\log\left(\frac{4N}{\xi}\right)
    <
    \log\left(\frac{2}{2-\xi}\right).
\end{equation}
This also ensures that
\[
    \frac{QJ}{\sigma^2}\log\left(\frac{4N}{\xi}\right)\geq1,
    \qquad\text{and}\qquad
    \frac{\sigma}{\sqrt J}\leq c_0.
\]

Applying Lemma \ref{lemma:admissible-M-window} with
\[
    \varepsilon=\beta=\frac{\xi}{2},
    \qquad
    r_J=r_J^\star,
    \qquad
    \mathcal G_J=\mathcal G_J^\star,
\]
gives an integer \(M\) satisfying
\begin{equation}
\label{eq:weight-M-window}
    \frac{6QJ}{\sigma^2}\log\left(\frac{4N}{\xi}\right)
    \leq
    M
    \leq
    \frac{7QJ}{\sigma^2}\log\left(\frac{4N}{\xi}\right).
\end{equation}
Moreover, the events
\[
    R_M^\star(J)
    \coloneqq
    \bigcap_{m=1}^M
    \{x_m\in\mathcal G_J^\star\},
\qquad 
    S_M(J)
    \coloneqq
    \left\{
        \left\|
            \left(\mbf G^{(M)}\right)^{-1}-\mbf I
        \right\|_2
        \leq
        \frac{\sigma}{\sqrt J}
    \right\},
\]
satisfy
\[
    \Prob[R_M^\star(J)]>1-\frac{\xi}{2},
    \qquad
    \Prob[S_M(J)]\geq1-\frac{\xi}{2}.
\]
Hence,
\[
    \Prob\left[
        R_M^\star(J)\cap S_M(J)
    \right]
    \geq
    \Prob[R_M^\star(J)]
    +
    \Prob[S_M(J)]
    -1
    >
    1-\xi.
\]

Fix a realization in this intersection. In the relative case,
\[
    \cos\gamma(x_m)>J^{-1/2},
    \qquad m=1,\ldots,M,
\]
so Corollary \ref{cor:deterministic-good-node-bounds} gives
\[
    \left|
        M w_m-\omega_{\mathcal L,\rho}(x_m)
    \right|
    \leq
    \sigma
    \left|
        \omega_{\mathcal L,\rho}(x_m)
    \right|.
\]
In the absolute case,
\[
    A_{\mathcal L,\rho}(x_m)<\sqrt J,
    \qquad m=1,\ldots,M,
\]
and the same corollary gives
\[
    \left|
        M w_m-\omega_{\mathcal L,\rho}(x_m)
    \right|
    \leq
    \sigma.
\]

Finally, \(S_M(J)\) implies that \(\mbf G^{(M)}\) is invertible.
Consequently, the least squares rule is well-defined and exact on \(V\).
\end{proof}

\begin{corollary}[Nonnegative intermediate rule]
\label{cor:positive-intermediate-rule}
Assume the hypotheses of \Cref{thm:finite-sample-weight-approximation}, let
\(\sigma\in(0,1)\) and \(0<\xi<1\), and suppose that
\((V,\mathcal L)\) is a nonnegative pair. Then there exist \(J\geq1\) and
an integer \(M\) satisfying
\[
    \frac{6QJ}{\sigma^2}\log\left(\frac{4N}{\xi}\right)
    \leq M \leq
    \frac{7QJ}{\sigma^2}\log\left(\frac{4N}{\xi}\right)
\]
such that, with probability at least \(1-\xi\), the least squares quadrature
rule is well-defined, exact on \(V\), and has real, nonnegative weights. If
\(L(x_m)>0\) at every sampled node, then all weights are strictly positive.
\end{corollary}

\begin{proof}
Apply \Cref{thm:finite-sample-weight-approximation} with the specified \(\xi\). On its event
of probability at least \(1-\xi\), the rule is well-defined, exact on \(V\),
and
\[
    \left| M w_m-\omega_{\mathcal L,\rho}(x_m) \right|
    \leq
    \sigma \left| \omega_{\mathcal L,\rho}(x_m) \right|.
\]
By Lemma \ref{lemma:real-ls-weights}, the weights are real. Since the pair is
nonnegative, \(\omega_{\mathcal L,\rho}(x_m)\geq0\), and hence
\[
    M w_m
    \geq
    (1-\sigma)\omega_{\mathcal L,\rho}(x_m)
    \geq0.
\]
Strict positivity follows whenever
\(\omega_{\mathcal L,\rho}(x_m)>0\), equivalently \(L(x_m)>0\).
\end{proof}

\subsection{Condition numbers of intermediate rules}
\label{subsec:intermediate-condition-numbers}

We now combine finite-sample weight approximation with the
reference-weight concentration condition in
\eqref{eq:reference-weight-concentration}. The first observation is
deterministic.

\begin{lemma}[Condition-number reduction]
\label{lemma:condition-number-reduction}
Suppose that
\[
    \left|
        Mw_m-\omega_{\mathcal L,\rho}(x_m)
    \right|
    \leq
    \sigma
    \left|
        \omega_{\mathcal L,\rho}(x_m)
    \right|,
    \qquad
    m=1,\ldots,M.
\]
Then,
\[
    \kappa_{\mathsf W}(\mathcal Q_M)
    \leq
    \frac{1+\sigma}{M}
    \sum_{m=1}^M
    \mathsf K_{\mathcal L,\rho,\mathsf W}(x_m).
\]
\end{lemma}

\begin{proof}
The coefficient bound in \eqref{eq:condition-number} and the triangle
inequality give
\[
    \kappa_{\mathsf W}(\mathcal Q_M)
    \leq
    \sum_{m=1}^M \frac{|w_m|}{\mathsf W(x_m)}
    \leq
    \frac{1+\sigma}{M}
    \sum_{m=1}^M
    \mathsf K_{\mathcal L,\rho,\mathsf W}(x_m).
\]
\end{proof}

\begin{remark}[Absolute weight control]
\label{rem:absolute-condition-number-reduction}
If instead
\[
    \left|
        Mw_m-\omega_{\mathcal L,\rho}(x_m)
    \right|
    \leq
    \sigma,
    \qquad
    m=1,\ldots,M,
\]
then
\[
    \kappa_{\mathsf W}(\mathcal Q_M)
    \leq
    \frac1M
    \sum_{m=1}^M
    \mathsf K_{\mathcal L,\rho,\mathsf W}(x_m)
    +
    \frac{\sigma}{M}
    \sum_{m=1}^M
    \frac1{\mathsf W(x_m)}.
\]
Thus, an analogous weighted stability result under absolute weight control
requires additional control of the empirical mean of
\(\mathsf W(X)^{-1}\).
\end{remark}

We next record the empirical concentration estimate used below.

\begin{lemma}[Reference-weight concentration]
\label{lemma:reference-weight-empirical-concentration}
Let \(X_1,\ldots,X_M\) be independent samples from \(\rho\), and set
\[
    Z_m
    \coloneqq
    \mathsf K_{\mathcal L,\rho,\mathsf W}(X_m).
\]
If
\(
    \left\|
        \mathsf K_{\mathcal L,\rho,\mathsf W}(X)
    \right\|_{\psi_1}
    <\infty
\)
for \(X\sim \rho\) then, for every \(t>0\),
\[
    \Prob\left[
        \left|
            \frac1M\sum_{m=1}^M Z_m
            -
            \left\|
                \frac{L}{\mathsf W}
            \right\|_{L^1_\mu(D)}
        \right|
        >
        t
    \right]
    \leq
    2\exp\left[
        -cM
        \min\left\{
            \frac{
                t^2
            }{
                \|
                    \mathsf K_{\mathcal L,\rho,\mathsf W}(X)
                \|_{\psi_1}^2
            },
            \frac{
                t
            }{
                \|
                    \mathsf K_{\mathcal L,\rho,\mathsf W}(X)
                \|_{\psi_1}
            }
        \right\}
    \right],
\]
where \(c>0\) is a universal constant. Consequently, for every
\(\eta\in(0,1)\), with probability at least \(1-\eta\),
\[
    \frac1M\sum_{m=1}^M Z_m
    \leq
    \left\|
        \frac{L}{\mathsf W}
    \right\|_{L^1_\mu(D)}
    +
    C
    \left\|
        \mathsf K_{\mathcal L,\rho,\mathsf W}(X)
    \right\|_{\psi_1}
    \left[
        \sqrt{\frac{\log(2/\eta)}{M}}
        +
        \frac{\log(2/\eta)}{M}
    \right],
\]
where \(C>0\) is universal.
\end{lemma}

\begin{proof}
The centered variables
\[
    Y_m
    \coloneqq
    Z_m-\E_\rho Z_m
\]
are independent and mean-zero, and satisfy
\[
    \|Y_m\|_{\psi_1}
    \leq
    C
    \left\|
        \mathsf K_{\mathcal L,\rho,\mathsf W}(X)
    \right\|_{\psi_1}.
\]
Bernstein's inequality for independent sub-exponential random variables
therefore gives the first estimate
\cite{vershynin2019high}. The second follows by inverting the tail bound and
using
\[
    \E_\rho Z_m
    =
    \E_\rho\left[
        \mathsf K_{\mathcal L,\rho,\mathsf W}(X)
    \right]
    =
    \left\|
        \frac{L}{\mathsf W}
    \right\|_{L^1_\mu(D)}
\]
from \eqref{eq:reference-weight-mean}.
\end{proof}

We restrict the following stability result to relative admissibility; under absolute weight control, an analogous result additionally requires concentration of \(\mathsf W(X)^{-1}\), as discussed in Remark~\ref{rem:absolute-condition-number-reduction}.

\stableIntermediateRule*
\begin{proof}
Let \(\mathcal G_J^{\mathrm{rel}}\) be the good-node set associated with
relative admissibility, so that
\[
    p_J
    =
    \rho(D\setminus\mathcal G_J^{\mathrm{rel}})
    \qquad\text{and}\qquad
    Jp_J\longrightarrow0.
\]
Choose
\[
    J\geq
    \max\left\{
        1,\frac{\sigma^2}{c_0^2}
    \right\}
\]
sufficiently large that \(p_J\leq\frac{\xi}{3}\) and
\begin{equation}
\label{eq:stable-window-choice}
    \frac{14QJp_J}{\sigma^2}
    \log\left(\frac{6N}{\xi}\right)
    <
    \log\left(\frac{3}{3-\xi}\right).
\end{equation}
Increasing \(J\) further if necessary, we may also ensure that
\[
    \frac{QJ}{\sigma^2}\log\left(\frac{6N}{\xi}\right)\geq1.
\]

Apply \Cref{lemma:admissible-M-window} with
\[
    \varepsilon=\beta=\frac{\xi}{3}.
\]
This gives an integer \(M\) satisfying \eqref{eq:stable-M-window} such that
\[
    \Prob[R_M^{\mathrm{rel}}(J)]>1-\frac{\xi}{3},
    \qquad
    R_M^{\mathrm{rel}}(J)
    =
    \bigcap_{m=1}^M
    \{x_m\in\mathcal G_J^{\mathrm{rel}}\},
\]
and
\[
    \Prob[S_M(J)]\geq1-\frac{\xi}{3},
    \qquad
    S_M(J)
    =
    \left\{
        \left\|
            \left(\mbf G^{(M)}\right)^{-1}-\mbf I
        \right\|_2
        \leq
        \frac{\sigma}{\sqrt J}
    \right\}.
\]

By \Cref{lemma:reference-weight-empirical-concentration} with $\eta=\xi/3$, the
event
\[
    E_{M,\mathsf W}
    =
    \left\{
        \frac1M
        \sum_{m=1}^M
        \mathsf K_{\mathcal L,\rho,\mathsf W}(x_m)
        \leq
        \left\|
            \frac{L}{\mathsf W}
        \right\|_{L^1_\mu(D)}
        +
        C
        \left\|
            \mathsf K_{\mathcal L,\rho,\mathsf W}(X)
        \right\|_{\psi_1}
        \left(
        \sqrt{\frac{\log(6/\xi)}{M}}
        +
        \frac{\log(6/\xi)}{M}
        \right)
    \right\}
\]
satisfies
\[
    \Prob[E_{M,\mathsf W}]\geq1-\frac{\xi}{3}.
\]

Consequently,
\begin{align*}
    \Prob\left[
        R_M^{\mathrm{rel}}(J)
        \cap S_M(J)
        \cap E_{M,\mathsf W}
    \right]
    &\geq
    \Prob[R_M^{\mathrm{rel}}(J)]
    +
    \Prob[S_M(J)]
    +
    \Prob[E_{M,\mathsf W}]
    -2 
    >
    1-\xi.
\end{align*}

Fix a realization in this intersection. The events
\(R_M^{\mathrm{rel}}(J)\) and \(S_M(J)\), together with
\Cref{cor:deterministic-good-node-bounds}, give
\[
    \left|
        Mw_m-\omega_{\mathcal L,\rho}(x_m)
    \right|
    \leq
    \sigma
    |\omega_{\mathcal L,\rho}(x_m)|,
    \qquad
    m=1,\ldots,M.
\]
Moreover, \(S_M(J)\) implies that the least squares rule is well-defined and
exact on \(V\). Applying Lemma~\ref{lemma:condition-number-reduction} and the
definition of \(E_{M,\mathsf W}\) gives
\[
    \kappa_{\mathsf W}(\mathcal Q_M)
    \leq
    (1+\sigma)
    \left[
        \left\|
            \frac{L}{\mathsf W}
        \right\|_{L^1_\mu(D)}
        +
        C
        \left\|
            \mathsf K_{\mathcal L,\rho,\mathsf W}(X)
        \right\|_{\psi_1}
        \frac{\log(6/\xi)}{\sqrt{M}}
    \right],
\]
which proves the result, where $C$ is a different universal constant than previously, and we have used $\log(6/\xi) > \sqrt{\log(6/\xi)}$ and $M \geq \sqrt{M}$.
\end{proof}

\begin{corollary}[Approximation consequence]
\label{cor:stable-intermediate-approximation}

Assume the hypotheses of
\Cref{thm:stable-intermediate-rule}, and suppose additionally that 
\[
    V \subset \mathcal{B}^\infty_{\mathsf W}(D), \qquad
    \mathcal L_{\mathsf W}\in (\mathcal{B}^\infty_{\mathsf W}(D))^*,
    \qquad
    \mathcal L_{\mathsf W}|_V=\mathcal L.
\]Then there exist \(J\geq1\) and an integer
\(M\) satisfying \eqref{eq:stable-M-window} such that, with
probability $1-\xi$, the least squares quadrature rule is exact on \(V\) and, for
every \(u\in \mathcal{B}^\infty_{\mathsf W}(D)\),
\begin{align*}
    &|\mathcal Q_M(u)-\mathcal L_{\mathsf W}(u)|
    \leq 
    \Bigg[
        (1+\sigma)
        \left(
            \left\|
                \frac{L}{\mathsf W}
            \right\|_{L^1_\mu(D)}
           + 
            C
            \left\|
                \mathsf K_{\mathcal L,\rho,\mathsf W}(X)
            \right\|_{\psi_1}
            \frac{\log(6/\xi)}{\sqrt{M}}
        \right)
        +
        \|\mathcal L_{\mathsf W}\|_{\infty,\mathsf W}^*
    \Bigg]
    e_{V,\infty,\mathsf W}(u),
\end{align*}
where \(X\sim\rho\).
\end{corollary}
\begin{proof}
Apply the weighted Lebesgue inequality of
Lemma \ref{lemma:lebesgue} to the rule supplied by
\Cref{thm:stable-intermediate-rule}.
\end{proof}

The preceding corollary is the principal quadrature error consequence of the
intermediate-rule construction: weighted stability converts approximation in
\(\mathcal{B}^\infty_{\mathsf W}(D)\) into a corresponding bound for
\(|\mathcal Q_M(u)-\mathcal L_{\mathsf W}(u)|\). In particular, any estimate for
\(e_{V,\infty,\mathsf W}(u)\), such as a weighted Jackson inequality, yields an
explicit quadrature error estimate.

The intermediate rule generally contains more than \(N=\dim(V)\) active nodes.
In the next section, we specialize to settings in which its weights are nonnegative and prune the rule while preserving exactness and nonnegativity.
When \(1\in V\), the resulting positive rule has the sharp unweighted condition number \( \kappa(\mathcal Q)=\mathcal L(1).\)

\section{Pruning exact quadrature rules}
\label{subsec:reduced}

Sections~\ref{sec:ls}--\ref{sec:posls} provide the input to the
pruning argument: on the event of probability at least \(1-\xi\) supplied by
Corollary \ref{cor:positive-intermediate-rule}, the least squares construction
produces an exact \(M\)-point rule with nonnegative weights. The present
section gives the deterministic second step. We show that any
nonnegative exact \(M\)-point rule can be compressed to a nested exact rule with
at most \(N=\dim(V)\) active nodes, without destroying positivity.  The
argument is a constructive Carath\'eodory--Steinitz reduction based only on
linear dependence in the moment equations.

\subsection{Positive pruning}

Let \(\{\varphi_n\}_{n=1}^N\) be a basis for \(V\), let
\(x_1,\ldots,x_M\in D\) be quadrature nodes, and let \(w_1,\ldots,w_M\) be quadrature weights. Define the Vandermonde matrix
\[
    \mathbf V_{m,n}=\varphi_n(x_m),
    \qquad (m,n)\in [M]\times [N],
\]
and the moment $N$-vector $\bs{\ell}$ as in \eqref{eq:moment-vector}.
Then, exactness on \(V\), i.e., $\mathcal L(v) = \mathcal{Q}_M(v)$ for every $v \in V$,
is equivalent to the moment system
\[
    \trans{\mathbf V}\mathbf w=\bs\ell,
    \qquad
    \mathbf w=\trans{(w_1,\ldots,w_M)}.
\]
Pruning is therefore the problem of reducing the support of a solution of
this linear system.  
We call the resulting quadrature rule with smaller support nested because
its nodes are a subset of the original $M$ nodes.
The positive case is the one needed for the constructive
Tchakaloff argument.

\begin{restatable}[Positive Carath\'eodory--Steinitz pruning]{theorem}{positiveCaratheodorySteinitzPruning}
\label{thm:caratheodory}
Assume that \(V\) is an \(N\)-dimensional real vector space of real-valued
functions.  Let \(\{x_m,w_m\}_{m=1}^M\), with \(w_m\ge0\), be a
nonnegative quadrature rule satisfying
\begin{equation}
\label{eq:L-exactness-assumption}
    \mathcal L(v)=\sum_{m=1}^M w_m v(x_m),
    \qquad v\in V.
\end{equation}
If \(M>N\), then there exists a nested quadrature rule of size $\widetilde M \leq N$, i.e. a subset
\(\{i_n\}_{n=1}^{\widetilde M}\subset [M]\), with
\(\widetilde M\le N\), and nonnegative weights
\(\{\widehat w_n\}_{n=1}^{\widetilde M}\subset [0,\infty)\) such that
\begin{equation}
\label{eq:L-exact-compressed}
    \mathcal L(v)=\sum_{n=1}^{\widetilde M}\widehat w_n v(x_{i_n}),
    \qquad v\in V.
\end{equation}
\end{restatable}
The proof is provided in \Cref{app:caratheodory} and reveals that the procedure to achieve \eqref{eq:L-exact-compressed} is computationally  constructive through identification of null vectors of $\trans{\mathbf V}$ and its submatrices. 
This type of node reduction is closely related to earlier work on positive cubature, moment compression, and efficient pruning procedures
\cite{Bos2016b,Bos2018b,Wilson1969,Cools1989,Rabinowitz1986,belik_efficient_2025}.
Different choices of null vector or
minimizing index in \eqref{eq:mstar_alphastar} may lead to different nested
rules, but all such choices preserve exactness and nonnegativity.

\begin{remark}[Stability of positive pruned rules]
\label{rem:positive-pruned-stability}
If \(1\in V\), then positive pruning preserves the sharp stability bound
  associated with exactness on constants.  In particular, with the integral functional $\mathcal L(v) = \int v d\mu$,
\[
    \mathcal Q_{\widetilde M}(v)
    =
    \sum_{n=1}^{\widetilde M}\widehat w_n v(x_{i_n}), \hskip 10pt
    \Longrightarrow \hskip 10pt
    \kappa(\mathcal Q_{\widetilde M})
    =
    \sum_{n=1}^{\widetilde M}\widehat w_n
    =
    \mathcal Q_{\widetilde M}(1)
    =
    \mathcal L(1).
\]
If further \(\mu\) is a probability
measure, then \(\kappa(\mathcal Q_{\widetilde M})=1\).
\end{remark}

With \Cref{thm:caratheodory}, the proofs for our main results in \Cref{ssec:preliminaries-main} are simple corollaries of our previously established results. We omit providing detailed proofs, here sketching only outlines:
\begin{itemize}
  \item Proof of \Cref{thm:finite-sample-weight-approximation-positive}: For every
  \(\sigma\in(0,1)\) and \(0<\xi<1\), Corollary
  \ref{cor:positive-intermediate-rule} gives an \(M\)-point nonnegative
  intermediate quadrature rule with probability at least \(1-\xi\).
  \item Proof of \Cref{thm:l2-tchakaloff}: With the non-negative quadrature rule from \Cref{thm:finite-sample-weight-approximation-positive}, apply \Cref{thm:caratheodory} to the real space \(V_{\mathbb R}\) of real dimension \(N\). Exactness extends to \(V = V_{\mathbb R} \oplus i V_{\mathbb R}\) by linearity.
  \item Proof of \Cref{thm:l2-tchakaloff-conjugate}: Apply
  \Cref{thm:l2-tchakaloff} to
  \((\mu,\rho,\widetilde V,\widetilde{\mathcal L})\). This yields a positive
  quadrature rule with at most
  \(P=\dim_{\mathbb C}\widetilde V\leq2N\) nodes that is exact on
  \(\widetilde V\), hence also on \(V\).
\end{itemize}

\subsection{Variants and interpretations}
We now record several ancillary but closely related variants and consequences of \Cref{thm:caratheodory} for quadrature.
\begin{corollary}[Algebraic pruning of exact rules]
\label{lem:algebraic-pruning}
Let \(\mathbb F \in \{\mathbb {R}, \mathbb{C}\}\).  Let \(V\) be an
\(N\)-dimensional vector space of \(\mathbb F\)-valued functions, and assume
that an \(M\)-point quadrature rule with weights \(w_m\in\mathbb F\)
satisfies
\[
    \mathcal L(v)=\sum_{m=1}^M w_m v(x_m),
    \qquad v\in V.
\]
If \(M>N\), then there exists a nested quadrature rule with at most \(N\)
nonzero weights that is exact on \(V\).
\end{corollary}

\begin{proof}
If any weight is zero, the corresponding node can be removed immediately.
We therefore assume that all active weights are nonzero.
Use the same moment system,
\[
    \trans{\mathbf V}\mathbf w=\bs\ell.
\]
Since \(M>N\), choose a nonzero
\(\mathbf k\in\ker(\trans{\mathbf V})\).  For any $m_\ast \in [M]$ such that \(k_{m_*}\neq0\), then setting
\[
    \alpha_*=\frac{w_{m_*}}{k_{m_*}}
\]
makes the \(m_*\)-th entry of \(\mathbf w-\alpha_*\mathbf k\) equal to zero,
while preserving the moment equations.  Repeating this procedure gives an exact
nested rule with at most \(N\) active nodes.
\end{proof}

This algebraic pruning step does not preserve positivity or stability. (Though with $\mathbb{F} = \mathbb{R}$, positivity could be preserved.)  In general the only role of this step is to show that interpolatory-size exactness is available whenever
arbitrary signed or complex weights are permitted.

\begin{corollary}
\label{cor:caratheodory_complex}
Assume that \(V\) is an \(N\)-dimensional complex-valued space and that an
\(M\)-point rule with real nonnegative weights is exact on \(V\).  Then there
exists a nested real nonnegative rule with at most \(2N\) active nodes that
is exact on \(V\).
\end{corollary}

\begin{proof}
Again, we begin by removing any nodes with zero weight.
Let \(\{\varphi_n\}_{n=1}^N\) be a basis for \(V\), let
\(\mathbf V\in\mathbb C^{M\times N}\) be the associated Vandermonde matrix,
let \(\mathbf w\in\mathbb R^M\) be the strictly positive weight vector, and let
\(\bs\ell\in\mathbb C^N\) be the moment vector.  Exactness is equivalent to
\[
    \trans{\mathbf V}\mathbf w=\bs\ell.
\]
Since \(\mathbf w\) is required to remain real, this complex system is
equivalent to the real system
\[
    \begin{bmatrix}
        \operatorname{Re}(\trans{\mathbf V}) \\
        \operatorname{Im}(\trans{\mathbf V})
    \end{bmatrix}
    \mathbf w
    =
    \begin{bmatrix}
        \operatorname{Re}(\bs\ell) \\
        \operatorname{Im}(\bs\ell)
    \end{bmatrix}.
\]
Equivalently, the real weights must match the moments of the real-valued
space
\[
    V_{\mathbb R}
    =
    \operatorname{span}_{\mathbb R}
    \{\operatorname{Re}\varphi_n,\operatorname{Im}\varphi_n:
      n=1,\ldots,N\}.
\]
This space has dimension at most \(2N\).  Applying
Theorem~\ref{thm:caratheodory} to \(V_{\mathbb R}\) gives a nested
nonnegative rule with at most \(2N\) active nodes, and matching the moments
of \(V_{\mathbb R}\) is equivalent to exactness on \(V\).
\end{proof}

\begin{remark}[Conjugate symmetry]
If the basis of \(V\) contains conjugate pairs, the real system in
Corollary~\ref{cor:caratheodory_complex} may contain redundant equations.
For example, suppose \(V\) has a basis of the form
\[
    \{\varphi_n\}_{n=1}^J
    \cup
    \{\overline{\varphi_n}\}_{n=1}^J
    \cup
    \{\psi_n\}_{n=1}^{N-2J},
    \qquad 2J\le N.
\]
Then the real and imaginary parts are spanned by
\[
    \{\operatorname{Re}(\varphi_n)\}_{n=1}^J
    \cup
    \{\operatorname{Im}(\varphi_n)\}_{n=1}^J
    \cup
    \{\operatorname{Re}(\psi_n)\}_{n=1}^{N-2J}
    \cup
    \{\operatorname{Im}(\psi_n)\}_{n=1}^{N-2J}.
\]
This set has cardinality at most
\[
    2J+2(N-2J)=2N-2J.
\]
Thus, in this case, the same argument gives a nested real nonnegative rule
with at most \(2N-2J\) active nodes. The number
\(P=\dim_{\mathbb C}\widetilde V\leq2N-2J\), with
\(\widetilde V=\mathrm{span}\{V,\overline V\}\).
\end{remark}

\section{Numerical example}
\label{sec:numerics}

We consider the standard Gaussian measure,
\[D = (-\infty,\infty),\quad d\mu(x) = \frac{1}{\sqrt{2\pi}}\exp\left(\frac{-x^2}{2}\right)dx.\]
Let $p_n(x)$ denote the degree-$n$ orthonormal Hermite polynomial. From this, we construct the following non-orthonormal basis:
\[
    \psi_1(x) = 1,\quad \psi_2(x) = \left|x+\frac{1}{2}\right|^a \left|x-\frac{1}{2}\right|^a,\quad \psi_3(x) = |x-1|^{-b} |x|^{-b}  |x+1|^{-b} , \quad \psi_{n}(x) = p_{n-3}(x),\quad n > 3,
\]
with $a,b>0$. We consider the following functional
\[\mathcal{L}(f) = \int_{-\infty}^\infty \psi_2(x) f(x) d\mu(x),\quad L(x) = \psi_2(x).\]
We note that the basis has singular behavior of order $b$ at $x=-1,0,1$ and the Riesz representer has isolated zeros of order $a$ at $x=-0.5,0.5$. To satisfy relative admissibility for induced sampling and mixture sampling \eqref{eq:mixed-sampling-density}, we choose $a = 0.49 < \frac{1}{2}$ and $b = 0.24<\frac{1}{4}$.

Using inner products evaluated with the adaptive integration package QuadGK.jl, we numerically orthonormalize the basis to obtain the trial basis \(\{\varphi_n\}_{n=1}^N\). The first five basis elements, along with the Riesz representer, are illustrated in the left panel of \Cref{fig:modified_hermite_basis_pdfs}.

To control the trial functions in the weighted supremum norm and ensure that \(\mathsf W^{-1}\in L_\mu^2(D)\), we choose
\[\mathsf{W}^{-2}(x) = \frac{k(x)}{N} = \frac{1}{N}\sum_{n=1}^N \varphi_n(x)^2.\]
Fixing $N=5$, we then visualize the density functions of the standard Gaussian, the induced distribution, and mixture distributions with $\theta=0.5$ and $\theta=0$ with respect to the Lebesgue measure in the right panel of \Cref{fig:modified_hermite_basis_pdfs}. We note that the case $\theta=0$ is not admissible as it results in $Q=\infty$.

\begin{figure}[htbp]
    \centering
    \includegraphics[width=0.48\textwidth]{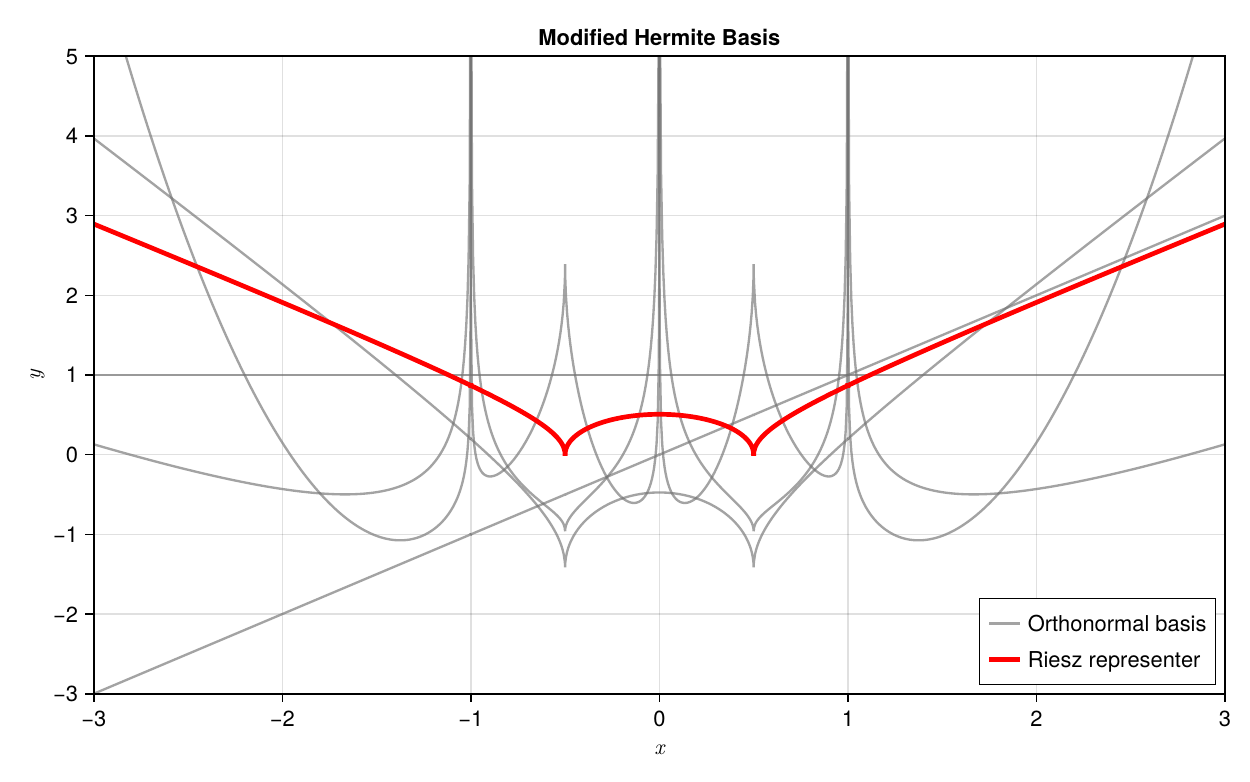}
    \includegraphics[width=0.48\textwidth]{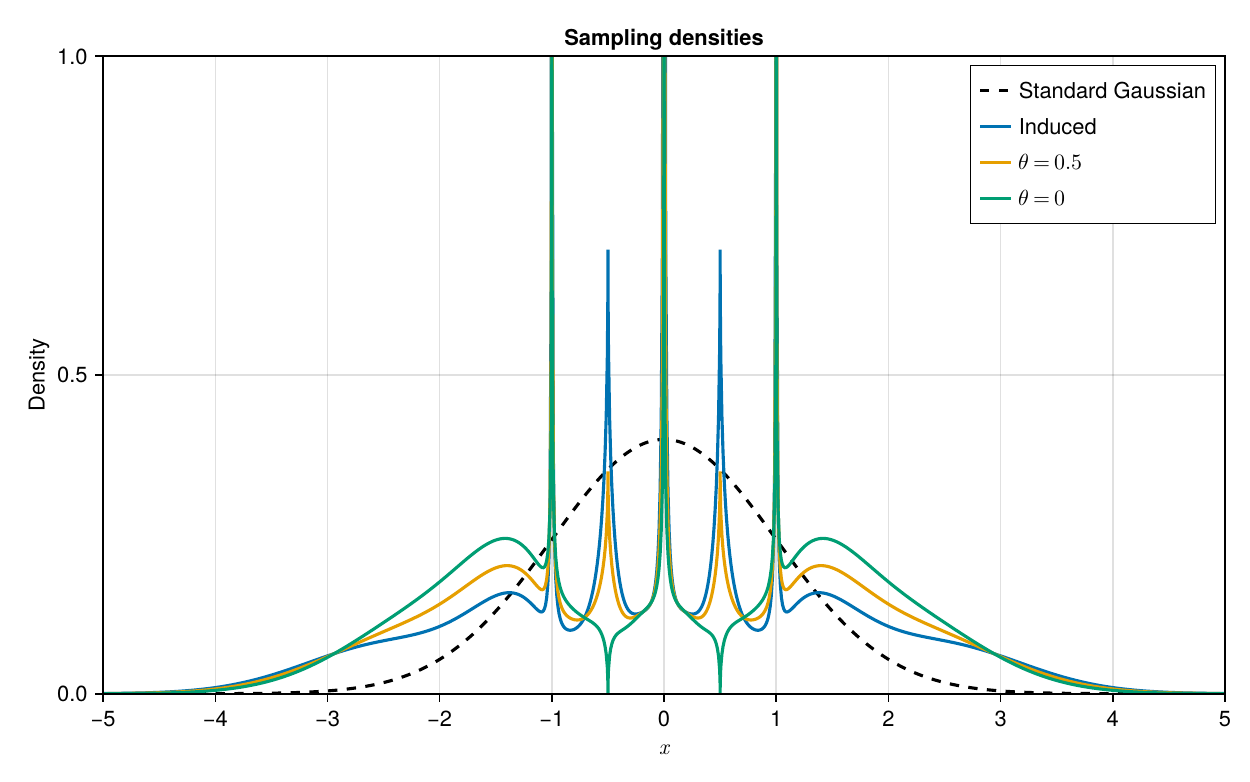}
    \caption{Left, orthonormal basis $\{\varphi_n\}_{n=1}^5$ along with the Riesz representer $\psi_2$. Right, the density functions of the standard Gaussian, the induced distribution, and the mixture distributions for $\theta=0.5$ and $\theta=0$, all with $N=5$.}
    \label{fig:modified_hermite_basis_pdfs}
\end{figure}

Because the basis is unbounded ($K=\infty$), standard sampling performs very poorly for this problem. Induced sampling provides us the optimal stability parameter, $Q=N$, however, still struggles due to the collapse of the limiting weights, $\omega_{\mathcal{L},\rho}(x)$ for $x$ near the singularities and of large magnitude. 

For simulations, we form the grid
\[
    N\in\{4,8,\dots,60\},
    \qquad
    M=\left\lceil \alpha N\log N\right\rceil,
    \qquad
    \alpha\in\{1.0,1.375,\dots,10.0\},
\]
with \(1000\) independent trials for each configuration. Samples from the induced and mixture distributions are generated by inverse-transform sampling using linear interpolation of a CDF evaluated on a fine grid.

\begin{figure}[htbp]
    \centering
    \includegraphics[width=0.32\textwidth]{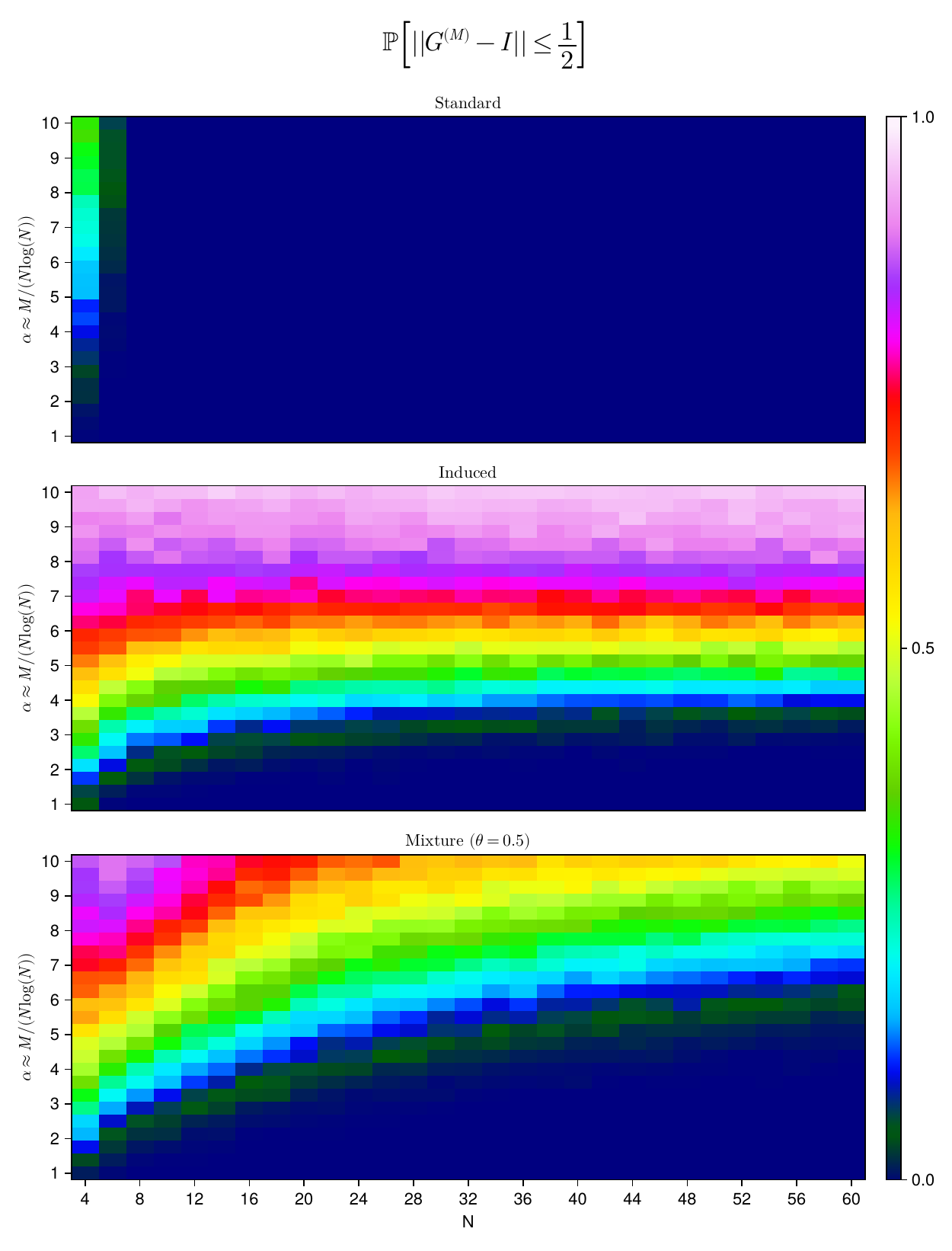}
    \includegraphics[width=0.32\textwidth]{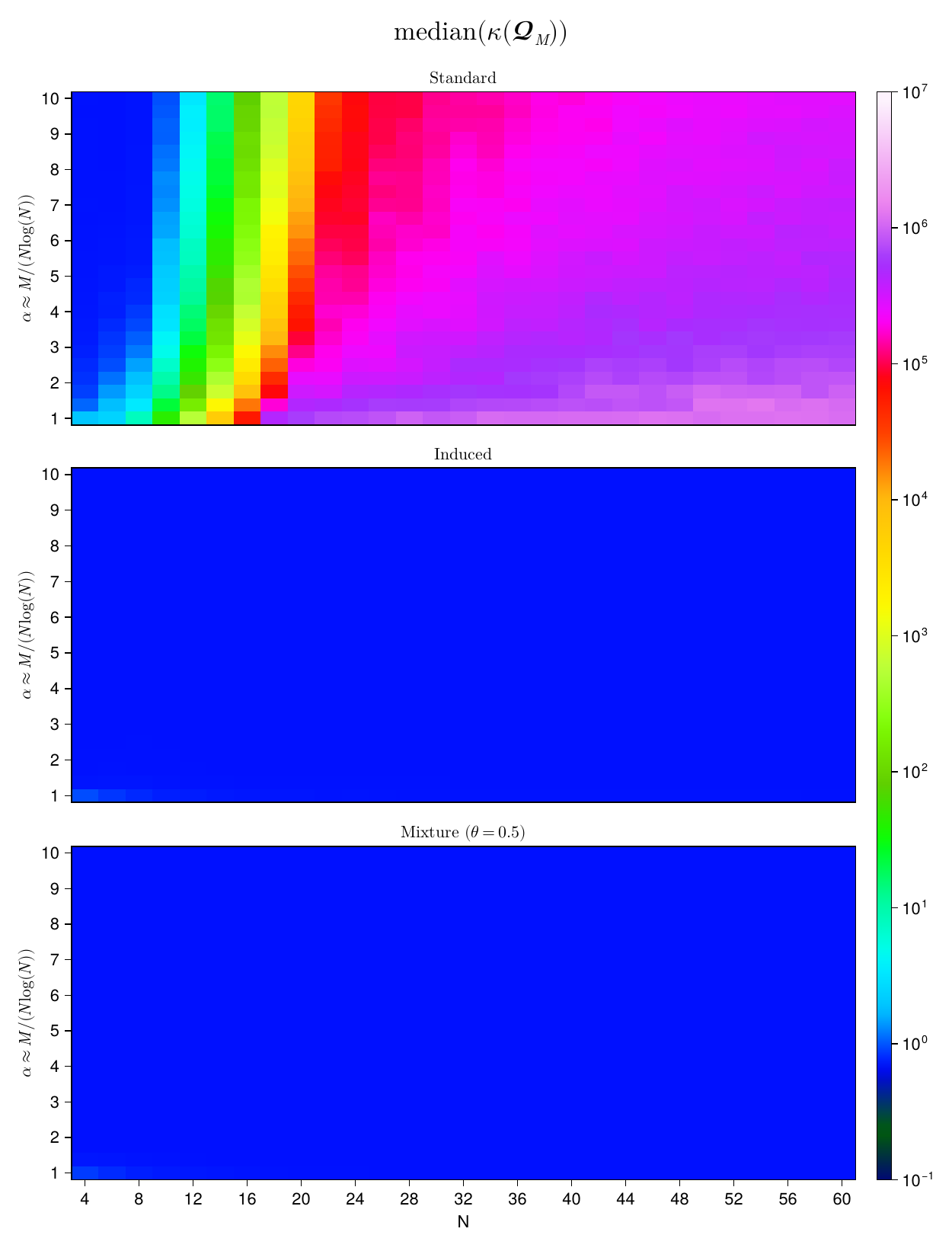}
    \includegraphics[width=0.32\textwidth]{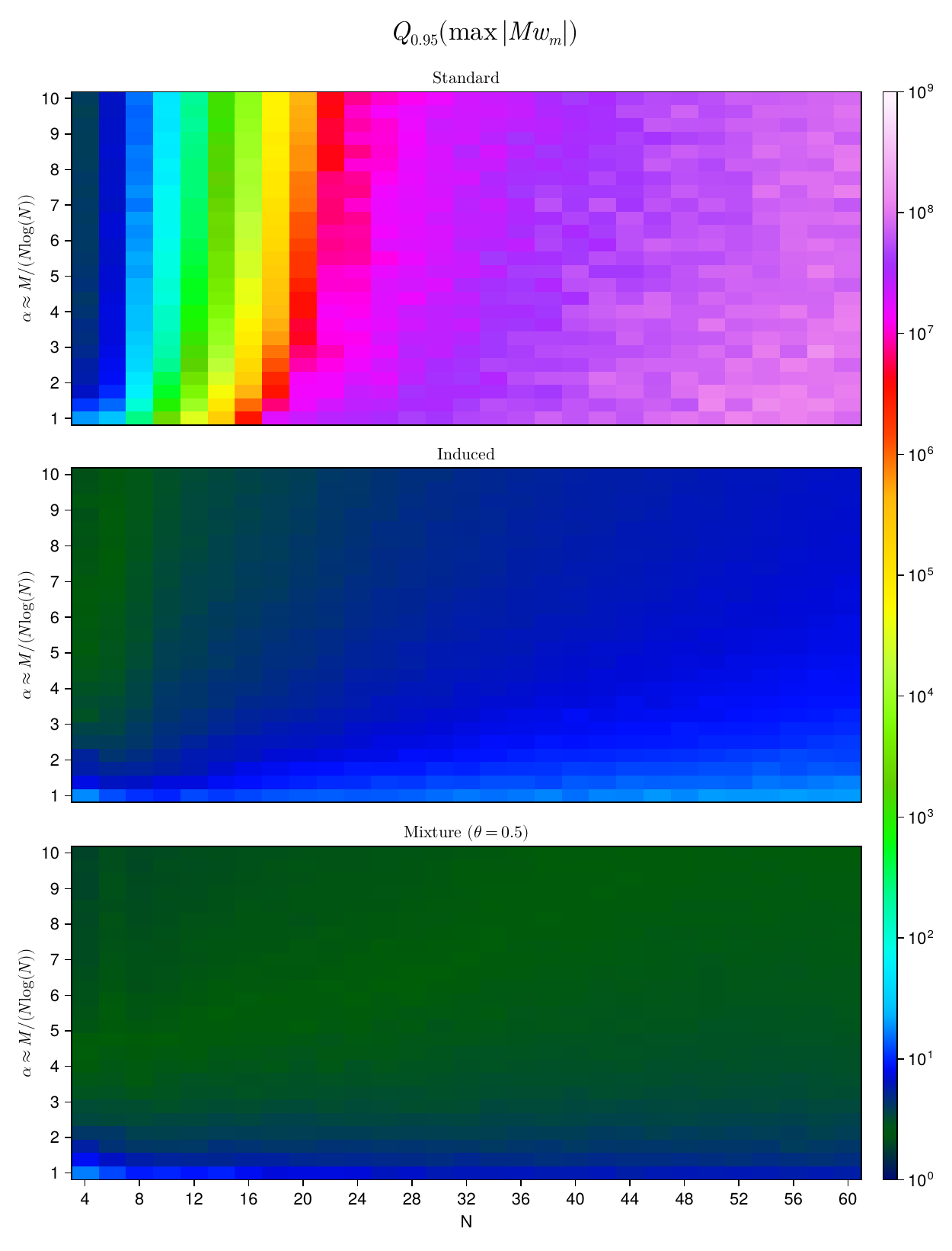}
    \caption{Left, empirical probability that $\|\mathbf{G}^{(M)} - \mathbf{I}\| \leq \frac{1}{2}$. Middle, empirical median condition number of the intermediate quadrature rule. Right, empirical 95th percentile of the maximum absolute weight rescaled by $M$. The rows correspond to standard sampling, induced sampling, and mixture sampling with $\theta=0.5$ from top to bottom respectively. Each square is computed from 1000 independent trials.}
    \label{fig:modified_hermite_heatmaps}
\end{figure}

\Cref{fig:modified_hermite_heatmaps} illustrates the numerical results for standard sampling, induced sampling, and mixture sampling with $\theta=0.5$. The left column shows that standard sampling rarely achieves Gramian concentration once \(N\geq8\). In agreement with the theory, induced sampling exhibits substantially stronger concentration on the scale \(M\asymp N\log N\). Mixture sampling with \(\theta=0.5\) incurs the predicted factor-of-two increase in the upper bound \(Q\leq N/\theta\), while retaining markedly better concentration than standard sampling. Over the range considered, its empirical concentration also exhibits some dependence on \(N\).

The middle column of \Cref{fig:modified_hermite_heatmaps} illustrates that, due to the difference in Gramian concentration, standard sampling results in significantly worse-conditioned intermediate quadrature rules than induced or mixture sampling. The conditioning of the induced- and mixture-sampled rules differs relatively little. The right column of \Cref{fig:modified_hermite_heatmaps} visualizes the 95th percentile of the maximum absolute weight rescaled by $M$. Clearly, due to poor conditioning, standard sampling has significantly larger weights in magnitude. Additionally, as expected by the theory, the mixture sampling reduces the magnitude of the largest weights.

Finally, \Cref{fig:modified_hermite_rules} shows representative strictly positive intermediate and pruned rules for \(N=10\) and \(M=500\). The intermediate weights obtained using induced and mixture sampling lie much closer to their corresponding limiting weights than do those obtained using standard sampling. The induced and mixture distributions place more mass in the tails, where \(k(x)\) is large. Relative to induced sampling, mixture sampling also reduces the largest limiting weights.

\begin{figure}[htbp]
    \centering
    \includegraphics[width=0.32\textwidth]{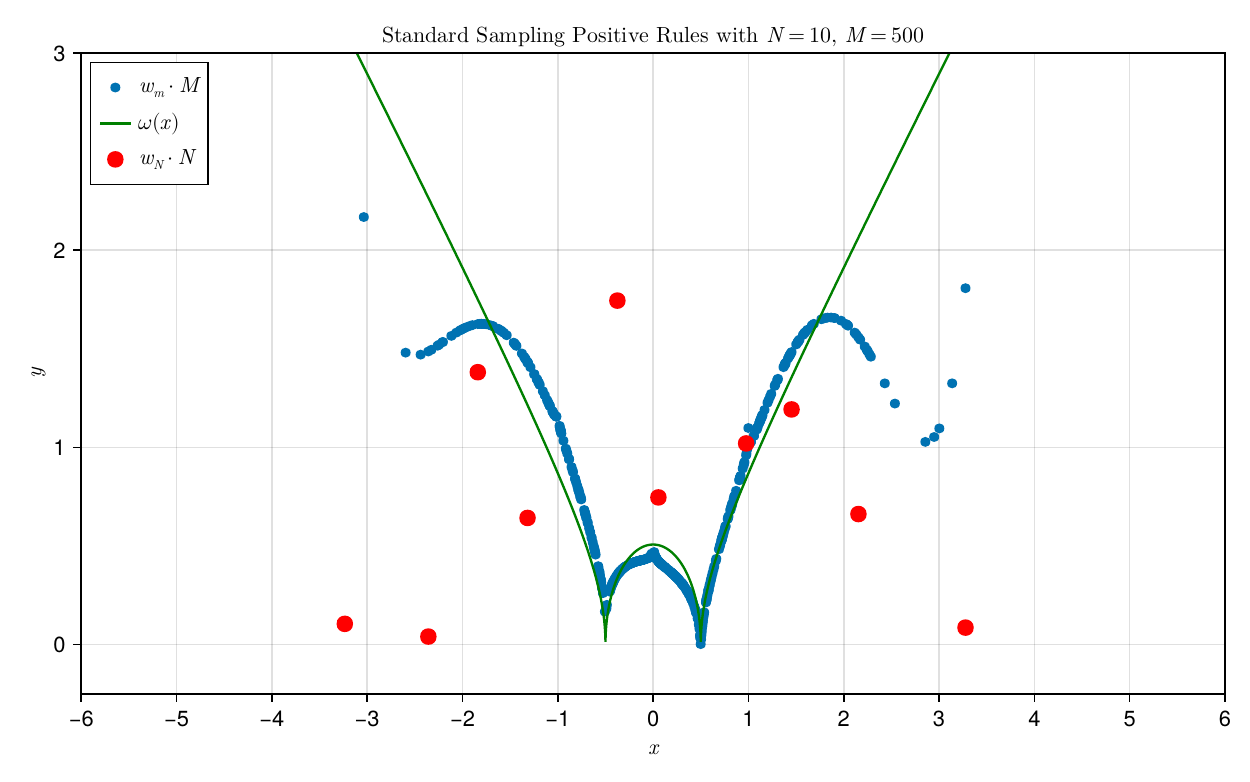}
    \includegraphics[width=0.32\textwidth]{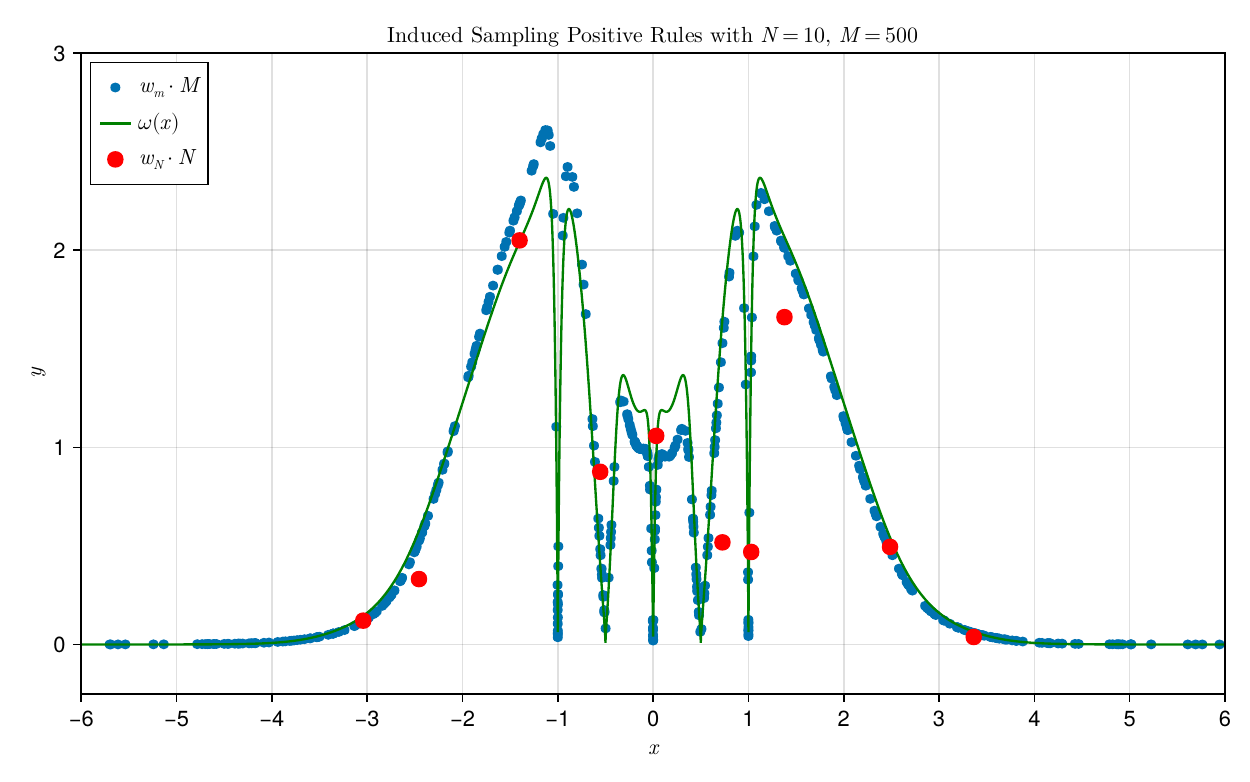}
    \includegraphics[width=0.32\textwidth]{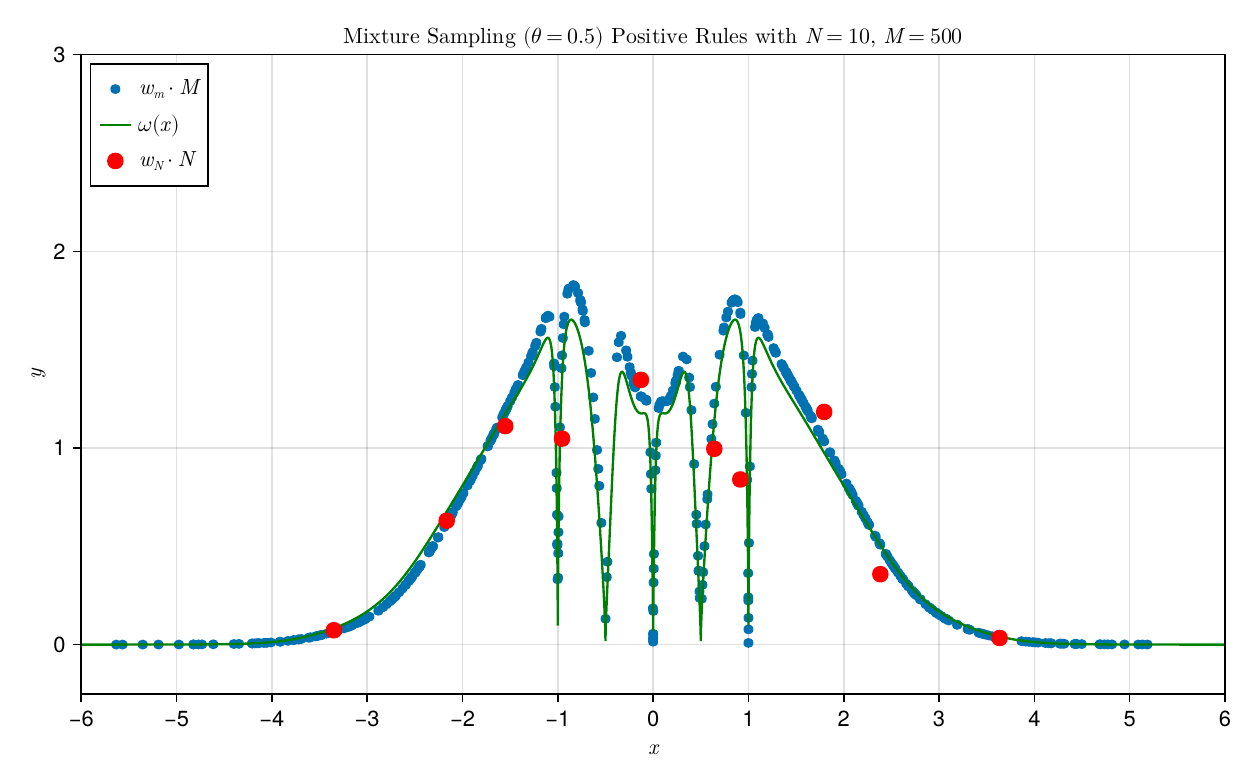}
    \caption{Examples of quadrature rules with $M=500$ and $N=10$ for (left) standard, (middle) induced, and (right) mixture sampling with $\theta=0.5$. The intermediate rules are strictly positive and are plotted alongside their limiting values. The strictly positive, $N$-point pruned rules are also plotted in red. Intermediate weights are scaled by \(M\) and pruned weights are scaled by \(N\).}
    \label{fig:modified_hermite_rules}
\end{figure}

Under induced and mixture sampling, the limiting weights approach zero near the zeros of the Riesz representer, near the singularities of the trial basis, and in the tails where \(k(x)\) is large. Under standard sampling, the limiting weight is simply \(L(x)\). This illustrates a difficulty with this approach for finding strictly positive rules: although induced and mixture sampling provide stronger Gramian concentration, they can also result in quadrature weights which concentrate very close to zero. 

We apply Carath\'eodory--Steinitz pruning to each positive intermediate rule. The procedure preserves exactness and nonnegativity and produces a nested rule supported on at most \(N\) nodes. The resulting pruned rules are plotted in red in \Cref{fig:modified_hermite_rules}.

\section{Disclosures}\label{sec:disclosures}

\subsection{Reproducibility Statement}
\label{ssec:reproducibility}

Software in Julia that reproduces the figures in \Cref{sec:numerics} is provided at \url{https://github.com/fbelik/ConstructiveTchakaloffRandomizedLeastSquares}. Seeds for pseudorandom number generators are fixed so that plots can be exactly reproduced.

\subsection{Data Availability Statement}
\label{ssec:data}
The data used to generate all plots in this submission are provided in the above repository. No external or extra data is required.

\subsection{Statement on the use of AI}
\label{ssec:ai-statement}

The authors used AI tools to proofread near-final drafts and identify typographical errors. They were also used as a repository cleanup tool for the provided software. The authors accept responsibility for all content and numerical results.

\subsection{Funding acknowledgements}
The authors acknowledge support from AFOSR award FA9550-23-1-0749 and NSF DMS-2136198.

\section{Conclusion}
\label{sec:conclusion}

We have analyzed a general randomized weighted least-squares construction for
exact and stable quadrature rules. Under small-ball relative admissibility
and reference-weight concentration conditions, we achieve, with high
probability, weighted stability and relative accuracy of the finite-sample
weights with respect to their limiting values. Our theory generalizes previously established uniform control on the Riesz representer $L$ and Christoffel function $k$, permitting unbounded values of $\sqrt{k}/L$. We have also introduced a sampling mixture strategy which achieves reference-weight concentration and preserves relative admissibility under suitable hypotheses. Under a nonnegative pair condition, relative concentration of the weights yields positive quadrature rules. Combining this with Carath\'eodory--Steinitz pruning yields interpolatory-size positive rules, giving a constructive generalized Tchakaloff result. We explore the specifics of our theory through a numerical example with unbounded values of $\sqrt{k}/L$, which reveals performance tradeoffs between different choices of samplers.

Future work may investigate extensions to Hilbert-valued integrands. 
Paired with currently allowed Banach-space domains $D$, this points to quadrature for vector-valued solution maps over parameter or function spaces. 
Several analytical questions nevertheless remain, such as
identification of space-dependent estimates for \(J\) under the minimal admissibility hypotheses. 
Another future direction could be to weaken the sub-exponential reference-weight concentration assumption. The examples in \Cref{ssec:assumptions-discussion} show that relative and absolute admissibility can hold even when the weighted reference magnitude has only polynomial tails, suggesting that useful high-probability stability estimates may remain available under weaker moment or tail assumptions.

\bibliographystyle{plainnat}
\bibliography{literature.bib}

@article{Cools1997,
  author =        {Cools, Ronald},
  journal =       {Acta Numerica},
  month =         jan,
  pages =         {1},
  publisher =     {Cambridge University Press (CUP)},
  title =         {Constructing cubature formulae: the science behind
                   the art},
  volume =        {6},
  year =          {1997},
  doi =           {10.1017/s0962492900002701},
  issn =          {1474-0508},
}

@Article{Cools1989,
  author    = {Cools, Ronald and Haegemans, Ann},
  title     = {On the construction of multi-dimensional embedded cubature formulae},
  journal   = {Numerische Mathematik},
  year      = {1989},
  volume    = {55},
  number    = {6},
  pages     = {735--745},
  month     = nov,
  issn      = {0945-3245},
  doi       = {10.1007/bf01389339},
  publisher = {Springer Science and Business Media {LLC}},
}

@book{Brass2011,
  author =        {Brass, Helmut and Petras, Knut},
  editor =        {Cohen, Ralph L. and Ellenberg, Jordan S. and
                   Singer, Michael A. and Sudakov, Benjamin},
  month =         oct,
  publisher =     {American Mathematical Society},
  series =        {Mathematical Surveys and Monographs},
  title =         {Quadrature Theory},
  volume =        {178},
  year =          {2011},
  doi =           {10.1090/surv/178},
}

@article{Golub1969,
  author =        {Golub, Gene H. and Welsch, John H.},
  journal =       {Mathematics of Computation},
  month =         may,
  number =        {106},
  pages =         {221--230},
  publisher =     {American Mathematical Society ({AMS})},
  title =         {Calculation of {Gauss} quadrature rules},
  volume =        {23},
  year =          {1969},
  doi =           {10.1090/s0025-5718-69-99647-1},
}

@article{Clenshaw1960,
  author =        {Clenshaw, C. W. and Curtis, A. R.},
  journal =       {Numerische Mathematik},
  month =         dec,
  number =        {1},
  pages =         {197--205},
  publisher =     {Springer Nature},
  title =         {A method for numerical integration on an automatic
                   computer},
  volume =        {2},
  year =          {1960},
  doi =           {10.1007/bf01386223},
}

@article{Novak1999,
  author =        {Novak, Erich and Ritter, Klaus},
  journal =       {Constructive Approximation},
  month =         oct,
  number =        {4},
  pages =         {499},
  publisher =     {Springer},
  title =         {Simple Cubature Formulas with High Polynomial
                   Exactness},
  volume =        {15},
  year =          {1999},
  doi =           {10.1007/s003659900119},
  issn =          {1432-0940},
}

@TechReport{Rabinowitz1986,
  author      = {Rabinowitz, P. and Kautsky, J. and Elhay, S. and Butcher, J.},
  title       = {On sequences of imbedded integration rules},
  institution = {The University of Adelaide},
  year        = {1986},
  type        = {Technical Report},
  number      = {86-02},
  month       = sep,
}

@article{Smolyak1963,
  author =        {Smolyak, S.},
  journal =       {Soviet Mathematics, Doklady},
  pages =         {240--243},
  title =         {Quadrature and interpolation formulas for tensor
                   products of certain classes of functions},
  volume =        {4},
  year =          {1963},
}

@article{Gerstner1998,
  author =        {Gerstner, Thomas and Griebel, Michael},
  journal =       {Numerical Algorithms},
  number =        {3/4},
  pages =         {209--232},
  publisher =     {Springer Nature},
  title =         {Numerical integration using sparse grids},
  volume =        {18},
  year =          {1998},
  doi =           {10.1023/a:1019129717644},
}

@article{Jakeman2017,
  author =        {Jakeman, John D. and Narayan, Akil},
  journal =       {Computer Methods in Applied Mechanics and
                   Engineering},
  month =         aug,
  pages =         {134--161},
  publisher =     {Elsevier {BV}},
  title =         {Generation and application of multivariate polynomial
                   quadrature rules},
  volume =        {338},
  year =          {2018},
  doi =           {10.1016/j.cma.2018.04.009},
}

@article{Keshavarzzadeh2018,
  author =        {Keshavarzzadeh, Vahid and Kirby, Robert M. and
                   Narayan, Akil},
  journal =       {ArXiV 1808.03707},
  title =         {Generation of Nested Quadrature Rules for Generic
                   Weight Functions via Numerical Optimization:
                   Application to Sparse Grids},
  year =          {2018},
}

@article{Keshavarzzadeh2018a,
  author =        {Keshavarzzadeh, Vahid and Kirby, Robert M. and
                   Narayan, Akil},
  journal =       {{SIAM} Journal on Scientific Computing},
  month =         jan,
  number =        {4},
  pages =         {A2033--A2061},
  publisher =     {Society for Industrial {\&} Applied Mathematics
                   ({SIAM})},
  title =         {Numerical Integration in Multiple Dimensions with
                   Designed Quadrature},
  volume =        {40},
  year =          {2018},
  doi =           {10.1137/17m1137875},
}

@article{Bos2018b,
  author =        {van den Bos, L. M. M. and Sanderse, B. and
                   Bierbooms, W. A. A. M. and van Bussel, G. J. W.},
  journal =       {ArXiV 1809.09842},
  title =         {Generating nested quadrature rules with positive
                   weights based on arbitrary sample sets},
  year =          {2018},
}

@article{Huybrechs2009,
  author =        {Huybrechs, Daan},
  journal =       {Journal of Computational and Applied Mathematics},
  month =         sep,
  number =        {2},
  pages =         {933--947},
  publisher =     {Elsevier BV},
  title =         {Stable high-order quadrature rules with equidistant
                   points},
  volume =        {231},
  year =          {2009},
  doi =           {10.1016/j.cam.2009.05.018},
  issn =          {0377-0427},
}

@article{Wilson1970,
  author =        {Wilson, M. Wayne},
  journal =       {Mathematics of Computation},
  number =        {110},
  pages =         {271--282},
  title =         {Discrete Least Squares and Quadrature Formulas},
  volume =        {24},
  year =          {1970},
}

@article{Bos2016b,
  author =        {van den Bos, L. M. M. and Koren, B. and
                   Dwight, R. P.},
  journal =       {Journal of Computational Physics},
  month =         mar,
  pages =         {418--445},
  publisher =     {Elsevier {BV}},
  title =         {Non-intrusive uncertainty quantification using
                   reduced cubature rules},
  volume =        {332},
  year =          {2017},
  doi =           {10.1016/j.jcp.2016.12.011},
}

@article{Wilson1969,
  author =        {Wilson, M. Wayne},
  journal =       {Mathematics of Computation},
  month =         may,
  number =        {106},
  pages =         {253--253},
  publisher =     {American Mathematical Society ({AMS})},
  title =         {A general algorithm for nonnegative quadrature
                   formulas},
  volume =        {23},
  year =          {1969},
  doi =           {10.1090/s0025-5718-1969-0242374-1},
}

@article{Bayer2006,
  author =        {Bayer, Christian and Teichmann, Josef},
  journal =       {Proceedings of the American mathematical society},
  number =        {10},
  pages =         {3035--3040},
  title =         {The proof of {Tchakaloff{\textquoteright}s} theorem},
  volume =        {134},
  year =          {2006},
}

@article{Tchakaloff1957,
  author =        {Tchakaloff, Vladimir},
  journal =       {Bulletin des Sciences Math{\'{e}}matiques},
  pages =         {123--134},
  title =         {Formules de cubatures m{\'{e}}caniques a coefficients
                   non n{\'{e}}gatifs},
  year =          {1957},
}

@Book{Jackson1982,
  title     = {Theory of Approximation (Colloquium Publications)},
  publisher = {American Mathematical Society},
  year      = {1982},
  author    = {Jackson, Dunham},
  isbn      = {978-0821810118},
}

@Article{Passow1970,
  author    = {Passow, Eli},
  title     = {Another proof of {Jackson's} theorem},
  journal   = {Journal of Approximation Theory},
  year      = {1970},
  volume    = {3},
  number    = {2},
  pages     = {146--148},
  month     = jun,
  doi       = {10.1016/0021-9045(70)90022-5},
  publisher = {Elsevier {BV}},
}

@article{Davis1967,
  author =        {Davis, Philip J.},
  journal =       {Mathematics of Computation},
  month =         oct,
  number =        {100},
  pages =         {578--582},
  publisher =     {American Mathematical Society ({AMS})},
  title =         {A Construction of Nonnegative Approximate
                   Quadratures},
  volume =        {21},
  year =          {1967},
  doi =           {10.1090/s0025-5718-1967-0222534-4},
}

@article{Ibrahimoglu2016,
  author =        {Ibrahimoglu, Bayram Ali},
  journal =       {Journal of Inequalities and Applications},
  month =         mar,
  number =        {1},
  pages =         {93},
  publisher =     {Springer Nature},
  title =         {Lebesgue functions and {Lebesgue} constants in
                   polynomial interpolation},
  volume =        {2016},
  year =          {2016},
  doi =           {10.1186/s13660-016-1030-3},
}

@book{Watson1980,
  author =        {Watson, G. A.},
  publisher =     {John Wiley \& Sons Ltd},
  title =         {Approximation Theory and Numerical Methods},
  year =          {1980},
  isbn =          {0471277061},
}

@article{cohen_stability_2013,
	title = {On the {Stability} and {Accuracy} of {Least} {Squares} {Approximations}},
	volume = {13},
	issn = {1615-3375, 1615-3383},
	doi = {10.1007/s10208-013-9142-3},
	number = {5},
	journal = {Foundations of Computational Mathematics},
	author = {Cohen, Albert and Davenport, Mark A. and Leviatan, Dany},
	month = oct,
	year = {2013},
	pages = {819--834},
}

@article{cohen_optimal_2017,
	title = {Optimal weighted least-squares methods},
	volume = {3},
	issn = {2426-8399},
	doi = {10.5802/smai-jcm.24},
	journal = {SMAI Journal of Computational Mathematics},
	author = {Cohen, Albert and Migliorati, Giovanni},
	year = {2017},
	note = {arxiv:1608.00512 [math.NA]},
	pages = {181--203},
}

@article{tropp_user-friendly_2012,
	title = {User-{Friendly} {Tail} {Bounds} for {Sums} of {Random} {Matrices}},
	volume = {12},
	issn = {1615-3375, 1615-3383},
	url = {http://link.springer.com/article/10.1007/s10208-011-9099-z},
	doi = {10.1007/s10208-011-9099-z},
	language = {en},
	number = {4},
	urldate = {2014-06-10},
	journal = {Foundations of Computational Mathematics},
	author = {Tropp, Joel A.},
	month = aug,
	year = {2012},
	pages = {389--434},
}

@misc{belik_efficient_2025,
    title = {Efficient and robust Carath\'{e}odory-Steinitz pruning of positive discrete measures},
    author = {B\v{e}l\'{i}k, Filip and Chan, Jesse and Narayan, Akil},
    year = {2025},
    eprint={2510.14916},
    archivePrefix={arXiv},
    primaryClass={math.NA},
    url={http://arxiv.org/abs/2510.14916}
}

@misc{schafer_2025,
      title={Beyond Tchakaloff Quadrature: Positive Functionals, Frames and Widths}, 
      author={Martin Schäfer and Tino Ullrich},
      year={2025},
      eprint={2511.15425},
      archivePrefix={arXiv},
      primaryClass={math.FA},
      url={https://arxiv.org/abs/2511.15425}, 
}

@article{migliorati2022stable,
  title={Stable high-order randomized cubature formulae in arbitrary dimension},
  author={Migliorati, Giovanni and Nobile, Fabio},
  journal={Journal of Approximation Theory},
  volume={275},
  pages={105706},
  year={2022},
  publisher={Elsevier}
}

@article{vershynin2019high,
  title={High-dimensional probability},
  author={Vershynin, Roman},
  journal={Cambridge Series in Statistical and Probabilistic Mathematics},
  volume={47},
  year={2019}
}

@article{lubinsky2007survey,
  title={A survey of weighted approximation for exponential weights},
  author={Lubinsky, Doron S},
  journal={arXiv preprint math/0701099},
  year={2007}
}

@book{dick_digital_2010,
  title = {Digital {{Nets}} and {{Sequences}}: {{Discrepancy Theory}} and {{Quasi-Monte Carlo Integration}}},
  shorttitle = {Digital {{Nets}} and {{Sequences}}},
  author = {Dick, Josef and Pillichshammer, Friedrich},
  year = 2010,
  edition = {First},
  publisher = {Cambridge University Press},
  address = {New York, NY},
  isbn = {978-0-521-19159-3}
}

@book{lemieux_monte_2009,
  title = {Monte {{Carlo}} and {{Quasi-Monte Carlo Sampling}}},
  author = {Lemieux, Christiane},
  year = 2009,
  series = {Springer {{Series}} in {{Statistics}}},
  publisher = {Springer New York},
  address = {New York, NY},
  doi = {10.1007/978-0-387-78165-5},
  isbn = {978-0-387-78164-8 978-0-387-78165-5}
}

@book{niederreiter_random_1992,
  title = {Random {{Number Generation}} and {{Quasi-Monte Carlo Methods}}},
  author = {Niederreiter, H.},
  year = 1992,
  series = {{{CBMS-NSF Regional Conference Series}} in {{Applied Mathematics}}},
  publisher = {{Society for Industrial and Applied Mathematics}},
  doi = {10.1137/1.9781611970081},
  isbn = {978-0-89871-295-7}
}

% Start appendices on separate page
\clearpage

\appendix

\section{Admissible subspaces}
\label{app:subspaces}

This appendix gives examples of tuples
\((\mu,\rho,V,\mathcal L)\) satisfying the relative admissibility condition
of \Cref{ssec:assumption}. Recall that
\[
    p_J
    =
    \rho\left(
        \cos\gamma\leq J^{-1/2}
    \right),
    \qquad
    \cos\gamma(x)
    =
    \frac{|L(x)|}
    {\|\mathcal L\|_{V^*}\sqrt{k(x)}},
\]
and that relative admissibility requires \(Q<\infty\) and
\[
    Jp_J\longrightarrow0.
\]

A necessary condition is
\[
    \rho\bigl(\{x\in D:L(x)=0,\ k(x)>0\}\bigr)=0,
\]
since \(\cos\gamma=0\) on this set. More generally, admissibility depends on
the \(\rho\)-mass of regions where \(L\) is small relative to \(\sqrt{k}\).
A convenient sufficient condition is
\[
    \frac{1}{\cos\gamma}
    =
    \frac{\|\mathcal L\|_{V^*}\sqrt{k}}{|L|}
    \in L^2_\rho(D);
\]
see \eqref{eq:relative-admissibility-sufficient}. Here, we consider three regimes:
\[
    \big(K<\infty,\quad L_{\inf}>0\big);
    \qquad\qquad
    \big(K<\infty,\quad L_{\inf}=0\big);
    \qquad\qquad
    \big(K=\infty,\quad L_{\inf}>0\big),
\]
where
\[
    K=\sup_{x\in D}k(x),
    \qquad
    L_{\inf}
    =
    \inf_{x\in D}
    \frac{|L(x)|}{\|\mathcal L\|_{V^*}}.
\]
The first case is immediate. In the second, we treat isolated algebraic
zeros of \(L\). In the third, we derive a sufficient Christoffel-tail
condition under induced sampling and verify it for spaces with isolated
singularities and for polynomial spaces under exponentially decaying
measures. This last regime includes integration on polynomial spaces over
unbounded domains, where \(L\equiv1\) when \(1\in V\).

\subsection{Bounded Christoffel function and a nonvanishing representer}
\label{ssec:finite-K-positive-Linf}

The first case is immediate: if \(k\) is uniformly bounded and \(L\) is
uniformly bounded away from zero, then \(\cos\gamma\) is uniformly separated from zero. This simplest setting admits sharper sample complexity bounds for Theorem \ref{thm:finite-sample-weight-approximation}.

\begin{lemma}
\label{lem:finite-K-positive-Linf}
Suppose that \(K<\infty\) and \(L_{\inf}>0\).
Then \((\mu,\rho,V,\mathcal L)\) is relatively admissible for every sampling measure \(\rho\) satisfying \(Q<\infty\). More precisely,
\[
    p_J=0
    \qquad\text{whenever}\qquad
    J>\frac{K}{L_{\inf}^2}.
\]
\end{lemma}

\begin{proof}
For every \(x\in D\),
\[
    \frac{1}{\cos\gamma(x)}
    \leq
    \frac{\sqrt K}{L_{\inf}}.
\]
Thus, \(p_J=0\) whenever \(J>K/L_{\inf}^2\), and consequently
\(Jp_J\to0\).
\end{proof}
\begin{remark}
In this case, one may, by \eqref{eq:J-treshold}, take
\[
    J>
    \max\left\{
        1,\,
        \frac{\sigma^2}{c_0^2},\,
        \frac{K}{L_{\inf}^2}
    \right\}.
\]
Hence, the sample size in \Cref{thm:finite-sample-weight-approximation}
satisfies
\[
    M
    =
    O\left(
        \frac{Q}{\sigma^2}
        \max\left\{
            1,\,
            \frac{\sigma^2}{c_0^2},\,
            \frac{K}{L_{\inf}^2}
        \right\}
        \log N
    \right).
\]
\end{remark}

\subsection{Bounded Christoffel function and isolated zeros}
\label{ssec:finite-K-zero-Linf}

We next allow \(L_{\inf}=0\), assuming that the zeros of \(L\) are
sufficiently isolated.

\begin{lemma}
\label{lem:finite-K-isolated-zero}
Suppose that \(K<\infty\), \(Q < \infty\), \(D\subset\mathbb R^d\), and 
that $L$ has an isolated zero near $\mathbf{c} \in \mathrm{int}(D)$ satisfying
\[
    \frac{|L(\x)|}{\|\mathcal L\|_{V^*}}
    \geq
    C\|\x-\mathbf{c}\|_2^\alpha
\]
for some \(C>0\) and \(\alpha>0\), and that \(\rho\) has a bounded
Lebesgue density near \(\mathbf{c}\). If
\[
    \alpha<\frac d2,
\]
then \((\mu,\rho,V,\mathcal L)\) is relatively admissible.
\end{lemma}

\begin{proof}
Denote the bad set by
\[
    E_J
    \coloneqq
    \left\{
        \x\in D:
        \cos\gamma(\x)\leq J^{-1/2}
    \right\}.
\]
Since \(L\) is bounded away from zero outside a neighborhood of \(\mathbf{c}\),
the set \(E_J\) lies in that neighborhood for all sufficiently large \(J\).
Moreover, using \(k\leq K\) and the local lower bound on \(L\),
\[
    E_J
    \subseteq
    \left\{
        \x\in D:
        \frac{|L(\x)|}{\|\mathcal L\|_{V^*}}
        \leq
        \sqrt{\frac{K}{J}}
    \right\}
    \subseteq
    B\left(\mathbf{c},C'J^{-1/(2\alpha)}\right)
\]
for some \(C'>0\). Since \(\rho\) has a bounded Lebesgue density near
\(\mathbf{c}\),
\[
    p_J
    =
    \rho(E_J)
    \lesssim
    J^{-d/(2\alpha)}.
\]
Therefore,
\[
    Jp_J
    \lesssim
    J^{1-d/(2\alpha)}
    \longrightarrow0
\]
whenever \(\alpha<d/2\).
\end{proof}

Note that in dimensions $d\leq2$, this excludes the possibility of $L$ having simple roots. 

\begin{remark}
The same conclusion holds for finitely many isolated zeros
\(c_1,\ldots,c_R\). Indeed, the bad set is contained in a finite union of
balls of radii \(O(J^{-1/(2\alpha_j)})\), so
\[
    p_J
    \lesssim
    \sum_{j=1}^R J^{-d/(2\alpha_j)}.
\]
Thus, relative admissibility holds whenever
\(\alpha_j<d/2\) for every \(j\).
\end{remark}

\subsection{Unbounded Christoffel function and a nonvanishing representer}
\label{ssec:infinite-K-positive-Linf}

We now assume \(L_{\inf}>0\) and sample from the induced measure
\[
    d\rho_{\mathrm{ind}}
    =
    \frac{k}{N}\,d\mu.
\]
For \(J>0\), define the Christoffel upper-tail set
\[
    A_J
    \coloneqq
    \{x\in D:k(x)\geq J\}.
\]
The examples below verify the sufficient condition
\begin{equation}
\label{eq:christoffel-tail-admissibility}
    J\int_{A_J}k(x)\,d\mu(x)
    \longrightarrow0
    \qquad\text{as }J\to\infty.
\end{equation}

\begin{lemma}
\label{lem:induced-tail-admissibility}
Suppose that \(L_{\inf}>0\) and
\eqref{eq:christoffel-tail-admissibility} holds. Then
\((\mu,\rho_{\mathrm{ind}},V,\mathcal L)\) is relatively admissible.
\end{lemma}

\begin{proof}
Since
\[
    \cos\gamma(x)
    \geq
    \frac{L_{\inf}}{\sqrt{k(x)}},
\]
we have
\[
    \{\cos\gamma\leq J^{-1/2}\}
    \subseteq
    A_{L_{\inf}^2J}.
\]
Therefore,
\[
    Jp_J
    \leq
    \frac{J}{N}
    \int_{A_{L_{\inf}^2J}}k(x)\,d\mu(x).
\]
Setting \(J\gets L_{\inf}^2J\), the right-hand side tends to zero by
\eqref{eq:christoffel-tail-admissibility}.
\end{proof}

We next give examples satisfying the Christoffel-tail condition
\eqref{eq:christoffel-tail-admissibility}.

\begin{example}[An isolated singularity]
\label{ex:singularity}
Let \(\mu\) be Lebesgue measure on a compact domain
\(D\subset\mathbb R^d\), and suppose that the functions in \(V\) are bounded
away from a point \(\mathbf{c}\in D\), while \(V\) contains functions with 
a singularity at $\mathbf{c}$ of largest strength
\[
    \|\x-\mathbf{c}\|_2^{-\alpha},
\]
with $\alpha>0$. Then \(K=\infty\). Nevertheless, if \(\alpha<d/4\), the Christoffel-tail
condition \eqref{eq:christoffel-tail-admissibility} holds. Notice that this
is stronger than the condition \(\alpha<d/2\) required for the singular
function to belong to \(L^2_\mu(D)\).
\end{example}

\begin{proof}
After translating the singularity, we may assume that \(\mathbf{c}=\mathbf{0}\) lies in the
interior of \(D\). Let \(V\) be spanned by
\(\{\psi_n\}_{n=1}^N\), where
\[
  \lim_{\x \to \mathbf{0}} \psi_1(\x) \| \x\|_2^\alpha = c_0 \neq 0,
\]
and suppose that the remaining functions do not have stronger singularities
at the origin. Thus, \(\psi_1\) behaves like \(r^{-\alpha}\), where
\(r=\|\x\|_2\). The condition \(\alpha<d/2\) ensures that
\(\psi_1\in L^2_\mu(D)\).

For sufficiently small \(\epsilon>0\), suppose that
\[
  \max_{n=1,\ldots,N}
  \sup_{\x\in D\setminus B(\mathbf{0},\epsilon)}
  |\psi_n(\x)|
  \leq C.
\]
Since every orthonormal basis of \(V\) is obtained by a linear change of
basis, it follows that
\[
  k(\x) \asymp \|\x\|_2^{-2\alpha}
\]
near the origin, while \(k\) is bounded away from the origin. Consequently,
for all sufficiently large \(T\),
\[
    A_T\subset B(\mathbf{0},C T^{-1/(2\alpha)})
\]
for some constant \(C>0\). Using spherical coordinates,
\begin{align*}
  T \int_{A_T} k(\x)\dx{\mu}(\x)
  &\leq
  T \int_{B(\mathbf{0},C T^{-1/(2\alpha)})} k(\x)\dx{\x} \\
  &\lesssim
  T \int_0^{C T^{-1/(2\alpha)}} r^{d-1-2\alpha}\dx{r}\\
  &\lesssim
  T^{1-(d-2\alpha)/(2\alpha)}.
\end{align*}
The exponent is negative precisely when \(\alpha<d/4\), and hence
\[
    T\int_{A_T}k(\x)\dx{\mu}(\x)\longrightarrow0.
\]

The same argument applies to a singularity at any interior point
\(\bs c\in D\), as well as to any finite number of singularities of this
form.
\end{proof}

\begin{example}
\label{ex:freud}
Suppose $D\subset\R^d$ and that \(\mu\) has a density satisfying
\[
    d\mu(\x)\le C\exp(-\|\x\|_2^\alpha)\,d\x,
    \qquad C,\alpha>0,
\]
possibly only for \(\|\x\|_2\) sufficiently large.  This includes, for
example, Gaussian-type measures.  Let \(V\) be any
finite-dimensional polynomial subspace.  If \(N>1\), then typically \(K=\infty\).  Nevertheless, \((\mu,\rho_{\mathrm{ind}},V,\mathcal L)\) is relatively admissible.
\end{example}

\begin{proof}
Suppose $\mu$ is a measure with a density satisfying
\begin{align*}
  \dx{\mu}(\x) &\leq C \exp(-\|\x\|_2^\alpha), & C, \alpha &> 0,
\end{align*}
for $\|\x\|_2$ sufficiently large.  Let $\x$ have coordinate representation 
\begin{align*}
  \x = (x^{(1)}, \ldots, x^{(d)}) \in \R^d,
\end{align*}
and let $p$ be the maximum (finite) polynomial degree of the polynomials in $V$. Without loss of generality we assume $p > 0$ (since if $p = 0$, then $K$ is finite).
Then since any orthonormal basis for $V$ can be expressed as a sum of monomials, we have
\begin{align*}
  k(\x) \leq \tilde{C} \left\| \x \right\|_2^{2p},
\end{align*}
where $\tilde{C}$ is finite but depends on $V$ and $d$. Now given some $T$ sufficiently large, we define
\begin{align*}
  R = \left(T/\tilde{C}\right)^{1/(2p)},
\end{align*}
so that $A_T \subset \R^d \setminus B(\mathbf{0},R)$. Thus,
\begin{align*}
  T \int_{A_T} k(\x) \dx{\mu}(\x) &\lesssim T \int_{\R^d \setminus B(\mathbf{0},R)} \|\x\|_2^{2p} \exp(-\|\x\|_2^\alpha) \dx{\x} \propto T \int_R^\infty r^{2p+d-1} \exp(-r^\alpha) \dx{r} \\
  &\propto T \int_{R^\alpha}^\infty t^{(2p+d-1)/\alpha} \exp(-t) t^{(1-\alpha)/\alpha} \dx{t} = T \, \Gamma\left(\frac{2p+d-\alpha}{\alpha} + 1, R^\alpha\right),
\end{align*}
where $\Gamma(\cdot,\cdot)$ is the upper incomplete Gamma function, defined as
\begin{align*}
  \Gamma(s,x) &\coloneqq \int_x^\infty t^{s-1} \exp(-t) \dx{t}.
\end{align*}
Using the facts that $R^\alpha \to \infty$ as $T \to \infty$ and $\Gamma(s,x)/(x^{s-1} e^{-x}) \to 1$ as $x \to \infty$, we have
\begin{align*}
  \lim_{T \to \infty} T \int_{A_T} k(\x)\dx{\mu}(\x) &\lesssim \lim_{T \to \infty} T \, \Gamma\left(\frac{2p+d-\alpha}{\alpha} + 1, R^\alpha\right) = \lim_{T \to \infty} T \left(R^\alpha\right)^{(2p+d-\alpha)/\alpha} \exp(-R^\alpha)  \\
&= \lim_{T \to \infty} T^{1 + (2p+d-\alpha)/(2p)} \exp(-T^{\alpha/2p}) = 0,
\end{align*}
confirming that this $(\mu, \rho_{\mathrm{ind}},V,\mathcal L)$ is relatively admissible.
\end{proof}

\section{Technical Lemmas}
\label{app:lemmas}

\subsection{Proof of Lemma \ref{lem:quant-sample-complexity}: Quantitative sample complexity}
\label{app:quant-sample-complexity}
\quantSampleComplexity*
\begin{proof}
Set
\[
    A
    \coloneqq
    \frac{Q}{\sigma^2}\log\left(\frac{4N}{\xi}\right).
\]
By Markov's inequality,
\[
    p_J
    =
    \rho\left(
        (\cos\gamma)^{-1}\geq\sqrt J
    \right)
    \leq
    C_rJ^{-r/2},
\]
and hence
\[
    AJp_J
    \leq
    AC_rJ^{-(r-2)/2}.
\]
Choose
\[
    J
    =
    \left(\frac{14 AC_r}{\zeta}\right)^{2/(r-2)},
\]
with $\zeta > 0$ such that 
\[
    14 AJp_J
    \leq
    \zeta.
\]
Hence, choosing $\zeta = \frac{\xi}{\alpha}$ where $\alpha\geq2$ satisfies $\zeta \in \left(0, \log\left(\frac{2}{2-\xi}\right)\right)$ will satisfy the 
upperbound requirement on $A J p_J$ of \Cref{lemma:admissible-M-window}.
Decreasing $\zeta$ further, i.e. increasing $\alpha$, increases $J$ allowing us to satisfy the lowerbound 
requirements of $A J$ and $J$.

Consequently, the relative finite-sample estimate holds for some
\[
    M
    \asymp
    AJ.
\]
Substituting the chosen value of \(J\) gives
\[
    M
    =
    O\left(
        A\left(\frac{AC_r}{\xi}\right)^{2/(r-2)}
    \right)
    =
    O\left(
        \left(\frac{C_r}{\xi}\right)^{2/(r-2)}
        A^{r/(r-2)}
    \right),
\]
which proves the result.
\end{proof}

\subsection{Proof of Lemma \ref{lem:real-preserving-representer}: Real-preserving representer}
\label{app:real-preserving-representer}
\realPreservingRepresenter*
\begin{proof}
Write
\[
    L=a+ib,
    \qquad
    a=\frac{L+\overline L}{2},
    \qquad
    b=\frac{L-\overline L}{2i}.
\]
Since \(V\) is closed under conjugation, \(a,b\in V_{\mathbb R}\).  If
\(\mathcal L\) is real-preserving, then for every \(v\in V_{\mathbb R}\),
\[
    \mathcal L(v)
    =
    \langle L,v\rangle
    =
    \int_D a(x)v(x)\,d\mu(x)
    -
    i\int_D b(x)v(x)\,d\mu(x)
\]
is real.  Hence,
\[
    \int_D b(x)v(x)\,d\mu(x)=0,
    \qquad v\in V_{\mathbb R}.
\]
Taking \(v=b\) gives \(\|b\|^2=0\).  Thus, \(b=0\) in \(L^2_\mu(D)\), so
\(L\) is real-valued as an element of \(L^2_\mu(D)\).

Conversely, if \(L\) is real-valued and \(v\in V_{\mathbb R}\), then
\[
    \mathcal L(v)
    =
    \langle L,v\rangle
    =
    \int_D L(x)v(x)\,d\mu(x)
    \in \mathbb R .
\]
Therefore, \(\mathcal L\) is real-preserving.
\end{proof}

\subsection{Proof of Lemma \ref{lemma:mixed-sampling-relative-admissibility}: Mixed sampling relative admissibility}
\label{app:mixed-sampling-relative-admissibility}
\mixedSamplingRelativeAdmissibility*
\begin{proof}
Let
\[
    E_\epsilon
    \coloneqq
    \{x\in D:\cos\gamma(x)\leq\epsilon\}.
\]
Since \(L\in L^2_\mu(D)\) and
\(\mathsf W^{-1}\in L^2_\mu(D)\), the Cauchy--Schwarz inequality gives
\[
    \left\|\frac{L}{\mathsf W}\right\|_{L^1_\mu(D)}
    \leq
    \|L\|_{L^2_\mu(D)}
    \|\mathsf W^{-1}\|_{L^2_\mu(D)}
    <\infty,
\]
so the mixed measure is well-defined. Moreover,
\eqref{eq:mixed-sampling-density} gives
\(Q_\theta\leq N/\theta<\infty\). Write
\[
    \rho_\theta
    =
    \theta\rho_{\mathrm{ind}}
    +
    (1-\theta)\rho_{\mathcal L,\mathsf W},
    \quad\text{where}\quad
    \dx\rho_{\mathcal L,\mathsf W}(x)
    =
    \frac{|L(x)|/\mathsf W(x)}
    {\|L/\mathsf W\|_{L^1_\mu(D)}}\,\dx\mu(x).
\]
Using
\(
    |L(x)|
    =
    \|\mathcal L\|_{V^*}\sqrt{k(x)}\cos\gamma(x),
\)
we obtain
\begin{align*}
    \rho_{\mathcal L,\mathsf W}(E_\epsilon)
    &\leq
    \frac{
        \epsilon\|\mathcal L\|_{V^*}
    }{
        \|L/\mathsf W\|_{L^1_\mu(D)}
    }
    \int_{E_\epsilon}
        \frac{\sqrt{k(x)}}{\mathsf W(x)}
    \,\dx\mu(x) \\
    &\leq
    \frac{
        \epsilon\|\mathcal L\|_{V^*}
        \|\mathsf W^{-1}\|_{L^2_\mu(D)}
    }{
        \|L/\mathsf W\|_{L^1_\mu(D)}
    }
    \left(
        \int_{E_\epsilon}k(x)\,\dx\mu(x)
    \right)^{1/2} \\
    &=
    \frac{
        \epsilon\|\mathcal L\|_{V^*}
        \sqrt N\,
        \|\mathsf W^{-1}\|_{L^2_\mu(D)}
    }{
        \|L/\mathsf W\|_{L^1_\mu(D)}
    }
    \rho_{\mathrm{ind}}(E_\epsilon)^{1/2}.
\end{align*}
By relative admissibility,
\(
    \rho_{\mathrm{ind}}(E_\epsilon)
    =
    o(\epsilon^2)
    \) for \( \epsilon\downarrow0,
\)
and hence
\[
    \rho_{\mathcal L,\mathsf W}(E_\epsilon)
    \lesssim
    \epsilon\,
    \rho_{\mathrm{ind}}(E_\epsilon)^{1/2}
    =
    o(\epsilon^2).
\]
Consequently,
\[
    \rho_\theta(E_\epsilon)
    =
    \theta\rho_{\mathrm{ind}}(E_\epsilon)
    +
    (1-\theta)
    \rho_{\mathcal L,\mathsf W}(E_\epsilon)
    =
    o(\epsilon^2),
\]
which proves relative admissibility. Finally, by \eqref{eq:mixed-sampling-density},
\[
    \tau_\theta^2(x)
    \geq
    (1-\theta)
    \frac{|L(x)|/\mathsf W(x)}
    {\|L/\mathsf W\|_{L^1_\mu(D)}},
\]
and therefore
\[
    \mathsf K_{\mathcal L,\rho_\theta,\mathsf W}(x)
    =
    \frac{|L(x)|}
    {\tau_\theta^2(x)\mathsf W(x)}
    \leq
    \frac{
        \|L/\mathsf W\|_{L^1_\mu(D)}
    }{
        1-\theta
    }.
\]
Thus,
\(\mathsf K_{\mathcal L,\rho_\theta,\mathsf W}(X)\) is uniformly bounded,
and hence sub-exponential, proving reference-weight concentration.
\end{proof}

\subsection{Proof of Lemma \ref{lemma:real-ls-weights}: Real least-squares weights}
\label{app:Real-ls-weights}
\realLeastSquaresWeights*
\begin{proof}
Since \(V\) is closed under complex conjugation, the real subspace
\[
    V_{\mathbb R}
    =
    \{v\in V: v=\overline v\}
\]
has complex span \(V\).  Choose an orthonormal basis
\(\{\psi_n\}_{n=1}^N\) of \(V_{\mathbb R}\).  Since the functions
\(\psi_n\) are real-valued, the corresponding evaluation matrix
\[
    (\mathbf \Psi)_{m,n}=\psi_n(x_m)
\]
is real.  Also, \(\mathbf W\) is real diagonal.  Therefore,
\[
    \mathbf G_\psi^{(M)}
    =
    \trans{\mathbf \Psi}\mathbf W\mathbf \Psi
\]
is a real symmetric matrix.  Since \(\mathcal L\) is real-preserving, the
moment vector
\[
    \bs\eta=\trans{(\eta_1,\ldots,\eta_N)},
    \qquad
    \eta_n=\mathcal L(\psi_n),
\]
belongs to \(\mathbb R^N\).  Hence, the weights computed in this real basis
are
\[
    \trans{\widehat{\mathbf w}}
    =
    \trans{\bs\eta}
    \left(\mathbf G_\psi^{(M)}\right)^{-1}
    \trans{\mathbf \Psi}\mathbf W
    \in \mathbb R^{1\times M}.
\]

It remains only to check that these are the same weights as those in
\eqref{eq:w-vector}.  Let \(\{\varphi_n\}_{n=1}^N\) be the complex
orthonormal basis used in \eqref{eq:w-vector}, and let \(\mathbf V\) be its
evaluation matrix.  Since both \(\{\varphi_n\}\) and \(\{\psi_n\}\) are
orthonormal bases of \(V\), there is a unitary matrix
\(\mathbf A\in\mathbb C^{N\times N}\) such that
\[
    \mathbf \Psi=\mathbf V\mathbf A .
\]
The corresponding moment vectors satisfy
\[
    \bs\eta=\trans{\mathbf A}\bs\ell,
    \qquad
    \trans{\bs\eta}=\trans{\bs\ell}\mathbf A,
\]
with $\bs{\ell}$ as in \eqref{eq:moment-vector}, and the Gramians satisfy
\[
    \mathbf G_\psi^{(M)}
    =
    \mathbf \Psi^*\mathbf W\mathbf \Psi
    =
    \mathbf A^*\mathbf G^{(M)}\mathbf A .
\]
Therefore,
\[
\begin{aligned}
    \trans{\widehat{\mathbf w}}
    &=
    \trans{\bs\eta}
    \left(\mathbf G_\psi^{(M)}\right)^{-1}
    \mathbf \Psi^*\mathbf W  \\
    &=
    \trans{\bs\ell}\mathbf A
    \left(\mathbf A^*\mathbf G^{(M)}\mathbf A\right)^{-1}
    \mathbf A^*\mathbf V^*\mathbf W  \\
    &=
    \trans{\bs\ell}
    \left(\mathbf G^{(M)}\right)^{-1}
    \mathbf V^*\mathbf W
    =
    \trans{\mathbf w} .
\end{aligned}
\]
Thus, \(\mathbf w=\widehat{\mathbf w}\in\mathbb R^M\).
\end{proof}

\subsection{Proof of Lemma \ref{lemma:G-chernoff}: Empirical Gram Concentration}
\label{app:G-chernoff}
\gramianConcentration*
\begin{proof}
  The following proof is in essence from \cite{cohen_stability_2013,cohen_optimal_2017}. Write
    \begin{align}
        \label{eq:GM-summands}
        \mbf G^{(M)}
        =
        \sum_{m=1}^M \mbf X_m,
        \qquad
        \mbf X_m
        =
        \frac{1}{M}
        \frac{
            \overline{\bs\varphi(x_m)}
            \trans{\bs\varphi(x_m)}
        }{\tau^2(x_m)} .
    \end{align}
    The matrices \(\{\mbf X_m\}_{m=1}^M\) are independent, Hermitian, and
    positive semidefinite.  Moreover,
    \[
        \E[\mbf X_m]=\frac{1}{M}\mbf I
    \]
    and
    \[
        \|\mbf X_m\|_2
        =
        \frac{1}{M}
        \frac{\|\bs\varphi(x_m)\|_2^2}{\tau^2(x_m)}
        =
        \frac{1}{M}
        \frac{k(x_m)}{\tau^2(x_m)}
        \leq
        \frac{Q}{M}
    \]
    almost surely.  Applying the matrix Chernoff bound
    \cite[Theorem~1.1]{tropp_user-friendly_2012} gives
    \begin{align*}
        \Prob\left[
            \lambda_{\min}(\mbf G^{(M)})<1-\delta
        \right]
        &\leq
        N\exp\left(-\frac{M c_-}{Q}\right),\\
        \Prob\left[
            \lambda_{\max}(\mbf G^{(M)})>1+\delta
        \right]
        &\leq
        N\exp\left(-\frac{M c_+}{Q}\right),
    \end{align*}
    where
    \[
        c_-
        =
        \delta+(1-\delta)\log(1-\delta),
        \qquad
        c_+
        =
        -\delta+(1+\delta)\log(1+\delta).
    \]
    For \(\delta\in(0,1)\), \(c_+\leq c_-\) and
    \(c_+\ge \delta^2/3\).  Hence,
    \begin{align*}
        \Prob\left[
            \left\|\mbf G^{(M)}-\mbf I\right\|_2>\delta
        \right]
        &\leq
        2N\exp\left(-\frac{M c_+}{Q}\right)
        \leq
        2N\exp\left(-\frac{M\delta^2}{3Q}\right).
    \end{align*}
    The lower bound \eqref{eq:Ghat-sampling-criterion} makes the final
    expression at most \(\epsilon\).
\end{proof}

\subsection{Proof of Theorem \ref{thm:caratheodory}: Carathéodory's theorem}
\label{app:caratheodory}
\positiveCaratheodorySteinitzPruning*
\begin{proof}
With the notation above, exactness is equivalent to
\[
    \trans{\mathbf V}\mathbf w=\bs\ell.
\]
If some weight is zero, the corresponding node may be removed immediately.
We therefore assume that all active weights are strictly positive.

Since \(M>N\), there exists a nontrivial real null vector
\[
    \mathbf k=\trans{(k_1,\ldots,k_M)}\in\ker(\trans{\mathbf V}).
\]
Hence, for every \(\alpha\in\mathbb R\),
\begin{equation}
\label{eq:caratheodory_downdate}
    \trans{\mathbf V}(\mathbf w-\alpha\mathbf k)
    =
    \trans{\mathbf V}\mathbf w
    =
    \bs\ell.
\end{equation}
Thus perturbing the weights along \(\mathbf k\) preserves exactness.

Choose
\begin{equation}
\label{eq:mstar_alphastar}
    m_*=\argmin_{m\in[M]:\,k_m\neq0}
    \left|\frac{w_m}{k_m}\right|,
    \qquad
    \alpha_*=\frac{w_{m_*}}{k_{m_*}}.
\end{equation}
Then
\[
    (\mathbf w-\alpha_*\mathbf k)_{m_*}=0.
\]
We claim that all other entries remain nonnegative.  If \(k_{m_*}>0\), then
\(\alpha_*>0\).  For every \(m\) with \(k_m>0\), the minimality of
\(|w_{m_*}/k_{m_*}|\) gives
\[
    \alpha_*=\frac{w_{m_*}}{k_{m_*}}
    \le
    \frac{w_m}{k_m},
\]
and hence \(w_m-\alpha_*k_m\ge0\).  For every \(m\) with \(k_m<0\), one has
\[
    w_m-\alpha_*k_m
    =
    w_m+\alpha_*|k_m|
    >
    0.
\]
The case \(k_{m_*}<0\) is symmetric.  Therefore,
\[
    \mathbf w-\alpha_*\mathbf k\ge0.
\]

By \eqref{eq:caratheodory_downdate}, the perturbed rule remains exact on
\(V\), and by construction at least one active weight has been set to zero.
Removing the zero-weight node gives a nonnegative exact rule with one fewer
active node.  Repeating this procedure until at most \(N\) active weights
remain gives \eqref{eq:L-exact-compressed}.
\end{proof}

\subsection{Integral Representer}
\label{app:integral-representer}
\begin{lemma}
\label{lem:integral-representer}
If \(\mathcal L\) is the integral functional $\mathcal L(v) = \int v d \mu $,
\(\mu\) is a probability measure, and \(1\in V\), then
\[
    L\equiv 1,
    \qquad
    L_{\inf}=1.
\]
\end{lemma}

\begin{proof}
Since \(1\in V\) and \(\mu\) is a probability measure, choose the
orthonormal basis so that \(\varphi_1\equiv1\).  Then,
\[
    \ell_n
    =
    \mathcal L(\varphi_n)
    =
    \int_D \varphi_n(x)\,d\mu(x)
    =
    \int_D \overline{\varphi_1(x)}\varphi_n(x)\,d\mu(x)
    =
    \delta_{1,n}.
\]
Hence,
\[
    L(x)
    =
    \sum_{n=1}^N \overline{\ell_n}\varphi_n(x)
    =
    1.
\]
The conclusion follows.
\end{proof}

\subsection{Weighted Lebesgue inequality}
\label{app:lebesgue}

\begin{lemma} 
\label{lemma:lebesgue} 
Suppose that
\[
V\subset \mathcal{B}^\infty_{\mathsf W}(D)
\qquad\text{and}\qquad
\mathcal L_{\mathsf W}\in \bigl(\mathcal{B}^\infty_{\mathsf W}(D)\bigr)^*,
\qquad
\mathcal L_{\mathsf W}|_V=\mathcal L,
\]
with
\[
\|\mathcal L_{\mathsf W}\|_{\infty,\mathsf W}^*
\coloneqq
\sup_{\substack{u\in \mathcal{B}^\infty_{\mathsf W}(D)\\ \|u\|_{\infty,\mathsf W}\leq1}}
|\mathcal L_{\mathsf W}(u)|.
\]
If \(\mathcal Q_M\) is exact on \(V\), then, for every \(u\in \mathcal{B}^\infty_{\mathsf W}(D)\), 
\[ 
|\mathcal Q_M(u)-\mathcal L_{\mathsf W}(u)| \leq \left( \kappa_{\mathsf W}(\mathcal Q_M) + \|\mathcal L_{\mathsf W}\|_{\infty,\mathsf W}^* \right) e_{V,\infty,\mathsf W}(u),
\] 
where 
\[ 
e_{V,\infty,\mathsf W}(u) = \inf_{v\in V} \|u-v\|_{\infty,\mathsf W}. 
\] 
\end{lemma} 
\begin{proof} 
For any \(v\in V\), exactness and the triangle inequality give 
\[ 
\begin{aligned} 
    |\mathcal Q_M(u)-\mathcal L_{\mathsf W}(u)| &= |\mathcal Q_M(u-v)-\mathcal L_{\mathsf W}(u-v)| \\
    &\leq |\mathcal Q_M(u-v)| + |\mathcal L_{\mathsf W}(u-v)| \\
    &\leq \left( \kappa_{\mathsf W}(\mathcal Q_M) + \|\mathcal L_{\mathsf W}\|_{\infty,\mathsf W}^* \right) \|u-v\|_{\infty,\mathsf W}.
\end{aligned} 
\] 
Here we used the definition of \(\kappa_{\mathsf W}(\mathcal Q_M)\). Taking the infimum over \(v\in V\) proves the result. 
\end{proof}

\end{document}